\documentclass[12pt]{article}

\usepackage{verbatim,color,amssymb}
\usepackage{amsmath}					
\usepackage{amsthm}		
\usepackage{bm}
\usepackage{dsfont}			
\usepackage{natbib}
\usepackage{multirow}
\usepackage{setspace}
\usepackage[mathscr]{euscript}
\usepackage{fancyhdr}
\usepackage{enumitem}
\usepackage{graphicx}
\usepackage{geometry}
\usepackage{xcolor}
\usepackage{hyperref}
\usepackage{bbm}

\usepackage{tabularx, booktabs}
\newcolumntype{Y}{>{\centering\arraybackslash}X}
\usepackage{lscape}

\usepackage{subfig}
\usepackage{float}
\usepackage{lineno}

\newtheorem{Thm}{\underline{\bf Theorem}}

\newtheorem*{Proof*}{Proof}

\newtheorem{Prop}{\underline{\bf Proposition}}
\newtheorem{Lem}{\underline{\bf Lemma}}
\newtheorem{Cor}{\underline{\bf Corollary}}

\def \R {\mathbb{R} }

\def \d {\mathrm{d}}

\def\log{\hbox{log}}

\def\P_25_ICML{{\it Proceedings of the 25th international conference on Machine learning}}

\def\bse{\begin{eqnarray*}}
	\def\ese{\end{eqnarray*}}
\def\be{\begin{eqnarray}}
	\def\ee{\end{eqnarray}}
\def\bq{\begin{equation}}
	\def\eq{\end{equation}}

\def\b1e{{\mathbf e}}
\def\b1f{{\mathbf f}}

\def\simind{\stackrel{\mbox{\scriptsize{ind}}}{\sim}}
\def\simiid{\stackrel{\mbox{\scriptsize{iid}}}{\sim}}

\renewcommand\footnoterule{\kern-3pt \hrule \textwidth 2in \kern 2.6pt}

\def\boxit#1{\vbox{\hrule\hbox{\vrule\kern6pt \vbox{\kern6pt \textcolor{blue}{#1}\kern6pt}\kern6pt\vrule}\hrule}}

\def\authorfootnote#1{{\let\thefootnote\relax\footnotetext{#1}}}

\allowdisplaybreaks

\begin{document}
	\thispagestyle{empty}
	\baselineskip=28pt

	\begin{center}
		{\LARGE{\bf Clustering consistency with\\ Dirichlet process mixtures
		}}
	\end{center}
	\baselineskip=12pt
	\vskip 24pt

	%\iffalse
	\small
	\baselineskip=14pt
	\begin{center}
	Filippo Ascolani$^{a,b}$ (filippo.ascolani@phd.unibocconi.it)\\
	Antonio Lijoi$^{a,b}$ (antonio.lijoi@unibocconi.it)\\
	Giovanni Rebaudo$^{b,c}$ (giovanni.rebaudo@austin.utexas.edu)\\
	Giacomo Zanella$^{a,b}$ (giacomo.zanella@unibocconi.it)
	
	%%%%%%%%%%%%%%%%%%%%%%%%%%%%%
	\vskip 3mm
	$^{a}$Department of Decision Sciences, Bocconi University,\\
	via R\"oentgen 1, 20136 Milan, Italy
	\vskip 4pt 
	$^{b}$Bocconi Institute for Data Science and Analytics, Bocconi University,\\ 
	Via Röntgen 1, 20136 Milan, Italy
	\vskip 4pt 
	$^{c}$Department of Statistics and Data Sciences,
	University of Texas at Austin,\\
	105 East 24th Street D9800, Austin, TX 78712, USA\\	
\end{center}
	\vskip 24pt 
	%\fi

\pagenumbering{arabic}
\setcounter{page}{0}
\newlength{\gnat}
\setlength{\gnat}{20pt}
\baselineskip=\gnat

\begin{center}
	{\Large{\bf Abstract}} 
\end{center}
\baselineskip=16pt
	Dirichlet process mixtures are flexible non-parametric models, particularly suited to density estimation and probabilistic clustering.
	In this work we study the posterior distribution induced by Dirichlet process mixtures as the sample size increases, and more specifically focus on consistency for the unknown number of clusters when the observed data are generated from a finite mixture. Crucially, we consider the situation where a prior is placed on the concentration parameter of the underlying Dirichlet process. Previous findings in the literature suggest that Dirichlet process mixtures are typically not consistent for the number of clusters if the concentration parameter is held fixed and data come from a finite mixture. Here we show that consistency for the number of clusters can be achieved if the concentration parameter is adapted in a fully Bayesian way, as commonly done in practice. Our results are derived for data coming from a class of finite mixtures, with mild assumptions on the prior for the concentration parameter and for a variety of choices of likelihood kernels for the mixture.
	
	\vskip 24pt 
	
\noindent\underline{\bf Key Words}: 
Asymptotic; Bayesian nonparametric; Consistency; Clustering; Dirichlet process mixture; Number of components. 
	
\par\medskip\noindent
\underline{\bf Short/Running Title}: Clustering consistency with DPM
	
\vskip 10mm	
\baselineskip=\gnat	

\section{Introduction}
Bayesian nonparametric methods have experienced a huge development in the last two decades, often standing out for their flexibility and coherent probabilistic foundations; see the monographs by \cite{Muller2017} and \cite{Subhashis2017} for recent stimulating accounts. 
The cornerstone of Bayesian nonparametrics is the model based on the Dirichlet process \citep{Ferguson1973b}, which can be expressed as
%\begin{equation}\label{DPmodel}
$X_i \mid  \tilde{P} \simiid \tilde{P}$ and $ \tilde{P}\sim \text{DP}(\alpha, Q_0)$, 
%\end{equation}
where $\alpha > 0$ is the \emph{concentration parameter} and $Q_0$ is the \emph{baseline distribution} over the sample space $(\mathbb{X}, \mathcal{X})$. 
The success of the Dirichlet process in actual implementations of the Bayesian approach to nonparametric problems is mostly due to its mathematical tractability, which is highlighted by conjugacy, and flexibility, which is assessed in terms of its large topological support. 

Since $\tilde P$ %The Dirichlet process 
is almost surely discrete, %and, 
if one wishes to model continuous data %, it is useful to 
one may convolve it with a %suitable 
density kernel $k$ parametrized by a latent variable $\theta$ that is drawn from a Dirichlet process. This %approach leads to 
yields the popular Dirichlet process mixture %model 
\citep{Lo1984b}, which  %where
%\begin{equation}\label{DPMmodel}
%X_i \mid \theta_i \simind k(\cdot|\theta_i), \quad \theta_i \mid \tilde{P} \simiid \tilde{P}, \quad \tilde{P} \sim \text{DP}(\alpha, Q_0).
%\end{equation}
%The Dirichlet process mixture model %in %\eqref{DPMmodel} 
exhibits appealing %nice 
asymptotic properties in the context of density estimation: in %mevery 
several relevant cases, the posterior distribution %of the unknown density 
concentrates at the true data-generating density %one 
at the minimax-optimal rate, up to a logarithmic factor, as the sample size increases \citep{Ghosal1999,Ghosal2007}.
Such a model and mevery of its variants are widely used across scientific areas, thanks also to the availability of a wide variety of efficient computational methods to perform inference, see for instance \cite{Escobar1995a,Escobar1998,Maceachern1998,Neal2000a,Blei2006a}.

Thanks to the discreteness of the Dirichlet process, the latent parameters $\theta_i$'s %will 
exhibit ties with positive probability. Hence, the Dirichlet process mixture model %in \eqref{DPMmodel} 
is also routinely used to perform clustering since it partitions observations into %clusters 
groups based on whether their corresponding latent parameters $\theta_i$ coincide or not. The ubiquitous use of Dirichlet process mixtures for clustering motivates the interest in the %it is important to understand 
asymptotic behaviour of the posterior distribution of the underlying partition, and in particular in 
%on 
the inferred number of clusters (i.e.\ subpopulations), as the number of observations increases. \cite{Nguyen2013} showed posterior consistency of the mixing distribution $\tilde{P}$ under general conditions. 
However, this does not imply consistency for the number of clusters, due to the use of the Wasserstein distance. Indeed, 
\cite{Miller2013} %provided a negative result, showing 
proved that Dirichlet process mixtures are not consistent for the number of components when data are generated from a mixture with a single standard normal component. 
See also \cite{Miller2014} for extensions.
These results, however, are derived under the assumption that the concentration parameter $\alpha$ is known and fixed.
This is crucial because the clustering behaviour of Dirichlet process mixtures is governed by the choice of $\alpha$. 
Indeed, under the Dirichlet process mixture model, %model\eqref{DPMmodel}, 
the prior probability of observing ties is %purely 
	a function solely of $\alpha$, since $\text{pr}( \theta_i = \theta_j ) = 1/(\alpha+1)$.

In order to have a more flexible distribution on the clustering of the data, in %several 
most implementations of the Dirichlet process mixture  %applications 
a prior $\pi$ for $\alpha$ is specified, leading to a mixing measure that is itself a mixture %mixture %of Dirichlet processes 
in the sense of \cite{Antoniak1974c}. Here we show that introducing such a prior %further
 has a major impact on the asymptotic behaviour of the number of clusters, as Dirichlet process mixtures can be consistent for the number of clusters. 
%In this paper we study the consistency of Dirichlet process mixtures under the commonly used assumption 
%of assigning a prior $\pi$ on the unknown concentration parameter. We show that Dirichlet process mixtures can be consistent for the number of clusters. 
We provide consistency results under fairly general conditions on $\pi$ and for a moderately large class of kernels $k$, including uniform and truncated normal distributions.
Following \cite{Miller2013}, we focus on data-generating mixtures with a single component. Our results also extend to the more general case of finite mixtures with multiple components, when a suitable separation assumption between the elements of the mixtures is fulfilled. Crucially, we prove consistency for cases where using a non-random $\alpha$ %results in 
yields inconsistency, thus suggesting that %placing 
a hyperprior may be beneficial even beyond the cases considered here.
We stress that the framework we study is arguably closer to the way Dirichlet process mixtures are used in practice, compared to holding $\alpha$ fixed.

We note that studying an asymptotic regime where the data-generating truth is a mixture with a finite and fixed number of components entails some degree of model misspecification. Indeed, Dirichlet process mixtures are nonparametric models with an %priori 
infinite number of components or, in other words, a number of clusters growing with the size of the dataset. Thus, our results can be interpreted as a form of robustness of the prior: if the number of components of the data-generating is finite, it can still be recovered by adapting appropriately the value of $\alpha$, despite the prior is concentrated on mixtures with infinitely mevery components. In particular we show that, under all the data generation mechanisms we consider in the next sections, the posterior distribution of $\alpha$ converges to a point mass at 0 at a specific rate, which is crucial to ensure consistency.
%Thus, our results can be interpreted as a form of robustness: despite assuming infinitely mevery components \textit{a priori}, if the data-generating truth has finitely mevery of them the model can recover that by adapting appropriately the value of $\alpha$. In particular we show that, under all the considered data generation mechanisms, the posterior distribution of $\alpha$ converges to a point mass at 0 at a specific rate, which is crucial to ensure consistency. 
See Section \ref{discussion} for more discussion and some related literature.

\section{Dirichlet process mixtures and random \mbox{partitions}}\label{sec:MM}
Henceforth, we will be focusing on Dirichlet process mixture models with a prior on the concentration parameter, namely %we obtain a mixture model of the following form
\begin{equation}\label{eq:MixDPM}
X_i|\theta_i \simind k(\cdot|\theta_i), \quad \theta_i \mid \tilde{P} \simiid \tilde{P}, \quad \tilde{P} \mid \alpha \sim \text{DP}(\alpha,Q_0), \quad \alpha \sim \pi,
\end{equation}
where $k(\,\cdot\,|\theta)$ is some density function, for every $\theta$.
Since we are interested in the distribution of the number of clusters, it is reasonable to rewrite \eqref{eq:MixDPM} in terms of the distribution on partitions, related to the so-called Chinese restaurant process. For every pair of natural numbers $(n, s)$ such that $s\le n$, denote with $\tau_s(n)$ the set of partitions of $\{1,\ldots,n\}$ into $s$ non empty subsets.
Conditionally on $\alpha$, the sequence $(\theta_i)_{i\ge 1}$ induces a prior distribution on the space of partitions of $\mathds{N}$ that, for every $n\ge 2$, is characterized by 

\begin{equation}\label{eq:EPPFDP}
\text{pr}(A \mid \alpha)=\dfrac{\alpha^s}{\alpha^{(n)}} \prod_{j=1}^s (a_j-1)!, \quad (A=\{A_1,\ldots,A_s\} \in \tau_s(n), s \le n),
\end{equation}
where $\alpha^{(n)} = \alpha \cdots (\alpha+n-1)$ is the ascending factorial and $a_j=|A_j|$ stands for the cardinality of set $A_j$. Conditionally on the partition $A$, the probability distributions of the data $X_{1:n}=(X_1,\dots,X_n)$ and of the cluster-specific parameters $\hat{\theta}_{1:s} = ( \hat{\theta}_1,\ldots,\hat{\theta}_s)$ %and of the associated observations  
are %given by

\begin{equation}\label{partition_model}
\text{pr}(X_{1:n} \mid \hat{\theta}_{1:s}, A) = \prod_{j = 1}^s\prod_{i \in A_j} k(X_i \mid \hat{\theta}_j ), \quad
\text{pr}( \hat{\theta}_{1:s} \mid A, \alpha) = \text{pr}( \hat{\theta}_{1:s} \mid A) = \prod_{j = 1}^s q_0(\hat{\theta}_j).
\end{equation}
The number of clusters in a sample of size $n$ is denoted by $K_n$ and under \eqref{eq:MixDPM} it has the following prior distribution 

\[
\text{pr}( K_n = s) = \int \sum_{A \in \tau_s(n)} \text{pr}( A \mid \alpha )\pi(\mathrm{d}\alpha).
\]
Since we %We 
are concerned with the large sample properties of %in studying 
$\text{pr}(K_n = s \mid X_{1:n})$, %so we are interested in 
we focus on the joint distribution of the vector $(X_{1:n}, K_n)$ which, for every $x_{1:n}=(x_1,\dots,x_n)\in \mathbb{X}^n$, is given by

\begin{equation}\label{first_step}
\text{pr}( X_{1:n}= x_{1:n}, K_n = s ) = \sum_{A \in \tau_s(n)} \text{pr}( A ) \prod_{j = 1}^s m (x_{A_j} ),
\end{equation} %where
where pr$(A)=\int \mbox{pr}(A|\alpha)\,\pi(\mathrm{d}\alpha)$ and
%\begin{equation*}%\label{eq:DPM_M}
$m(x_{A_j}) = \int \prod_{i \in A_j} k(x_i \mid \theta) q_0(\theta) \mathrm{d}\theta$
%\end{equation*}
is the marginal likelihood for the subset of observations identified by $A_j$, given that they are clustered together. We study the asymptotic behaviour of the posterior induced by model \eqref{eq:MixDPM} when the observations are independent and identically distributed samples from a finite mixture, that is we assume the following data generation mechanism
\begin{equation}\label{true}
X_i \simiid P = \sum_{j = 1}^t p_jR_j, \quad (i = 1, 2,\dots),
\end{equation}
where, for every $t\ge 1$, the $R_j$'s are distinct %is a 
probability measures on $\mathbb{X}$ %for every $j$, $t \geq 1$, 
and the $p_j$'s are probability weights, i.e. $p_j\in (0,1)$ for every $j$ and $\sum_j p_j=1$. %probability weights. 
We will let $P^{(n)}$ and $P^{(\infty)}$ be the product probability measures induced on $\mathbb{X}^n$ and $\mathbb{X}^\infty$ respectively, and denote \eqref{true} by $X_{1:\infty}\sim P^{(\infty)}$. % for brevity.
In the following, we will consider each $R_j$ to be dominated by a suitable measure and denote the resulting density by $f_j(\cdot) := f(\cdot \mid \theta_j^\ast)$. We say that model in \eqref{eq:MixDPM} is \emph{well-specified} for $P$ if $k(\cdot|\theta)=f(\cdot \mid \theta)$, that is if the data-generating distribution is a mixture of kernels belonging to the same parametric family that defines %the prior in used to to define
%belongs to the family of kernels in 
\eqref{eq:MixDPM}.

We say that posterior consistency for the number of clusters holds if $\text{pr}(K_n = t \mid X_{1:n}) \to 1$ as $n\to\infty$ in $P^{(\infty)}$-probability. Note that %here 
the conditional probability $\text{pr}(K_n = t \mid X_{1:n})$ is defined with respect to the model in \eqref{eq:MixDPM}, while the convergence in probability is with respect to the data-generating process $X_{1:\infty}\sim P^{(\infty)}$. Since $\text{pr}(K_n = t \mid X_{1:n})$ lies between 0 and 1, convergence in $P^{(\infty)}$-probability is equivalent to convergence in $L^1$ with respect to $P^{(\infty)}$ and thus we could equivalently define consistency in terms of $L^1$ convergence.

\section{Main consistency results}\label{sec:consistency_results}
%Under the setting of \eqref{eq:MixDPM}, we will derive consistency results with priors on $\alpha$ that satisfy the following requirements:
The investigation of the asymptotics of the number of clusters $K_n$, induced by the model in \eqref{eq:MixDPM}, will rely on the following assumptions on the prior $\pi$ of $\alpha$
\begin{enumerate}
\item[$A1.$] \emph{Absolute continuity}: $\pi$ is absolutely continuous with respect to the Lebesgue measure and its density is still denoted as $\pi$;
\item[$A2.$] \emph{Polynomial behaviour around the origin}: $\exists \, \epsilon $, $\delta$, $\beta$ such that $\forall \alpha \in (0, \epsilon)$ it holds $\frac{1}{\delta}\alpha^\beta \leq \pi(\alpha) \leq \delta \alpha^\beta$;
\item[$A3.$] \emph{Subfactorial moments}: $ \exists \, D, \nu, \rho > 0$ such that $\int \alpha^s \pi(\alpha) \, d\alpha < D\rho^{-s}\Gamma(\nu+s+1)$ for every $s\ge 1$. 
\end{enumerate}
The first two assumptions are sufficient to study the posterior moments of $\alpha$, conditional to the number of groups $K_n$, as will be clarified in Proposition \ref{prop:Prior}. Assumption $A3$, instead, will be useful specifically for consistency purposes: the minimum value of $\rho$ %needed
required to achieve consistency depends on the problem at hand, that is on the specific choice of $P$ in \eqref{true} and $k$ in \eqref{eq:MixDPM}, as will be stated in Theorems \ref{th:ConsUnif} and \ref{th:ConsNormDelta}. Assumptions $A1$-$A3$ are satisfied by common families of distributions, as displayed in the next lemma.
\begin{Lem}\label{priors}
The following choices of $\pi$ satisfy assumptions $A1$, $A2$ and $A3$ (for a fixed $\rho > 0$)
\begin{itemize}
\item[{\rm (1)}] every distribution with bounded support that satisfies assumptions $A1$ and $A2$, such as the uniform distribution over $(0,c)$, with $c>0$;
\item[{\rm (2)}] The Generalized Gamma distribution with density proportional to $\alpha^{d-1}e^{-\left(\frac{\alpha}{a} \right)^p}$, provided that $p > 1$; 
\item[{\rm (3)}] The Gamma distribution with shape $\nu$ and rate $\rho$.
\end{itemize}
\end{Lem}
Note that the rate parameter of the Gamma distribution corresponds to the quantity $\rho$ in assumption $A3$.

\subsection{General consistency result for location families\\with bounded support}\label{sec:Bound}
For our general result we consider kernels of the form 
\begin{equation}\label{data_generating_general}
\begin{aligned}
k(x \mid \theta) &= g(x - \theta)%\mathbbm{1}_{[\theta-c , \theta +c ]}(x)
 &(x\in\mathbb{R}),
\end{aligned}
\end{equation}
where $c > 0$ and $\theta\in\mathbb{R}$ is a location parameter. Here $g$ is a density function on the real line satisfying the following assumptions
\begin{enumerate}
	\item[$B1.$] $g$ is strictly positive on some interval $[a,b]$
	 and 0 elsewhere;
	\item[$B2.$] $g$ is differentiable with bounded derivative in $(a,b)$; 
	\item[$B3.$] The base measure $Q_0$ is absolutely continuous with respect to the Lebesgue measure, and its density $q_0$ is bounded. % with bounded density $q_0$. 
\end{enumerate}
The above assumptions essentially require that the kernel is a location-family distribution with positive density on a bounded support. The class is fairly general and it includes, as relevant special cases, the uniform distribution and the truncated Gaussian distribution, among others.

When considering a mixture of the kernels in \eqref{data_generating_general} as data generation mechanism satisfying $B1$--$B3$, with true parameters $\theta^* = (\theta_1^*, \dots, \theta_t^*)$, we say that $\theta^*$ is \emph{completely separated} if $|\theta_j^*-\theta_k^*| > b-a$, for every $j \neq k$. This assumption is somewhat restrictive, but sufficient to prove that the addition of a prior on $\alpha$ may solve the inconsistency issue. Indeed, we have the following general consistency result.
\begin{Thm}\label{th:ConsBounded}
Suppose $k$ and $q_0$ satisfy assumptions $B1$--$B3$. If $\pi$ satisfies assumptions $A1$--$A3$ with $\rho$ high enough then, for every $P$ as in \eqref{true} with $t \in \{1, 2, \dots\}$, $f_j=k(\cdot|\theta^*_j)$, $\theta^*$ completely separated and $\theta^*_j$ belonging to the interior support of $Q_0$ for every $j$, we have

\[
\text{pr}(K_n = t \mid X_{1:n}) \to 1 %\quad \text{as } n \to \infty
\]
as $n \to \infty$ in $P^{(\infty)}$-probability. On the contrary, if %instead 
$\pi(\alpha) = \delta_{\alpha^*}(\alpha)$, with $\alpha^* > 0$, then
\[
\lim \sup \text{pr}(K_n = t \mid X_{1:n}) < 1
\]
as $n \to \infty$ in $P^{(\infty)}$-probability.
\end{Thm}
As discussed above, the minimum value of $\rho$ needed depends on the specific function $g$ and prior distribution $Q_0$. Therefore, a prior on the concentration parameter yields consistency when the true data generating distribution meets a condition of complete separability, that informally amounts to having cluster locations sufficiently distinct. Note that this condition is automatically satisfied when $t=1$. We additionally show that, even under such an assumption, the Dirichlet process mixture model with fixed $\alpha$ still fails to be consistent at the number of clusters. 
Hence, a prior on $\alpha$ %helps overcoming
is crucial to overcome issues with learning the true number of clusters as the sample size increases.

Moreover, the posterior mass on a smaller number of clusters than the truth vanishes, as explained in the next proposition. The latter holds under mild assumptions on model \eqref{eq:MixDPM}, satisfied either by bounded distributions as above or for instance by the Gaussian kernel.

\newpage
\begin{Prop}\label{undersmoothing}
Let $P$ be as in \eqref{true}, with true parameters $\theta^*_1, \dots, \theta^*_t$. Let $\theta_j^*$ belong to the support of $Q_0$ for every $j = 1, \dots, t$ and let $k$ satisfy assumptions $B1$--$B3$ above or $H1$--$H4$ in the supplementary material. Then
\begin{equation}
\text{\rm pr}(K_n < t \mid X_{1:n}) \to 0
\end{equation}
 in $P^{(\infty)}$-probability as $n \to \infty$.
\end{Prop}

\subsection{Consistency on specific examples}
Theorem \ref{th:ConsBounded} requires $\rho$ in assumption $A3$ to be high enough, depending on the specific formulation of the model. In order to provide an example, we focus on the case of uniform kernel and $t = 1$, that is
\begin{equation}\label{unif_framework}
f = \text{Unif}(\theta^\ast-c, \theta^\ast+c), \quad k(\cdot| \theta) = \text{Unif}(\theta-c, \theta+c), \quad q_0 = \text{Unif}(\theta^\ast-c, \theta^\ast+c),
\end{equation}
where $\theta^\ast \in \mathbb{R}$ is a fixed location parameter and $c>0$. In this setting the marginal distribution is available and with a suitable application of H\"{o}lder's inequality one can %it is possible to 
prove consistency for specific values of $\rho$.

\begin{Thm}\label{th:ConsUnif}
Consider $f$, $k$ and $q_0$ as in \eqref{unif_framework}, and assume $\pi$ satisfies %assumptions
$A1$--$A3$ (with $\rho \ge 38$). Then
\[
\text{pr}(K_n = 1 \mid X_{1:n}) \to 1% \quad \text{as } n \to \infty
\]
as $n \to \infty$ in $P^{(\infty)}$-probability.
\end{Thm}
As a second example, we move beyond bounded kernels and consider a simple, yet interesting, case. More precisely, we specialize model \eqref{eq:MixDPM} to %the case of 
Gaussian kernels and assume constant data, equal to some fixed real number $\theta^\ast$, setting
\begin{equation}\label{eq:NormDelta}
f%f(\theta^\ast) 
= \delta_{\theta^\ast}, \quad k(\cdot| \theta) = \text{N}(\theta,1), \quad q_0 = \text{N}(0, 1).
\end{equation}
Unlike the other examples, this case is not well-specified, as $k(\cdot|\theta)\neq f(\cdot)$ for every $\theta$. This makes the definition of true or data-generating number of clusters more delicate. Nonetheless, being an example with constant data, one would hope the posterior of the number of clusters to concentrate on one cluster. 
However, even in such a limiting case, % with constant data, 
\cite{Miller2013} %interestingly 
show that under \eqref{eq:MixDPM} with fixed concentration parameter $\text{pr}(K_n=1|X_{1:n})$ does not converge to $1$ as $n$ diverges.

Once again, placing a prior on $\alpha$ %changes 
impacts the posterior asymptotic behaviour of $K_n$ and one achieves consistency, as detailed in the next theorem.
\begin{Thm}\label{th:ConsNormDelta}
Consider $(f, k, q_0)$ as in \eqref{eq:NormDelta} and assume $\pi$ satisfies %assumptions
$A1$--$A3$ (with $\rho > 16$). Then
\[
\text{pr}(K_n = 1 \mid X_{1:n}) \to 1 
\] 
$P^{(\infty)}$-almost surely as $n\to \infty$.
\end{Thm}
Finally, note that %all 
the previous consistency results are related to another property %result 
of general interest, namely %which is that 
the posterior distribution of the concentration parameter converges to a point mass at $0$, %goes to zero 
if posterior consistency for the number of clusters holds.
\begin{Prop}\label{prop:PostAlpha}
%Under every of the settings in Theorems \ref{th:ConsBounded}, \ref{th:ConsUnif}, and \ref{th:ConsNormDelta},
Let the data be generated as in \eqref{true} with $t \in \mathbb{N}$ and assume $\pi$ satisfies $A1$ and $A2$. Then if $\text{pr}(K_n = t \mid X_{1:n}) \to 1$
we have
\[
\pi(\alpha \mid X_{1:n}) \rightarrow \delta_0 %\quad \text{as } n \to \infty 
\]
weakly, as $n \to \infty$, in $P^{(\infty)}$-probability.
\end{Prop}
Hence, under the conditions that ensure consistency for the number of clusters, the posterior distribution of the concentration parameter converges to a degenerate distribution at $0$. 
This is not surprising since the Dirichlet process mixture model is concentrated on mixtures with infinitely mevery components and one way to achieve consistency is to let $\alpha$ tend to zero, which entails that %since it is as if 
the prior is swamped by the data.

\section{Methodology and proof technique}\label{sec:Methodology}
\subsection{The role of the prior on the concentration parameter}
Our proofs of consistency in Theorems \ref{th:ConsBounded}, \ref{th:ConsUnif} and \ref{th:ConsNormDelta} rely on the following lemma.
\begin{Lem}\label{sufficient}
The convergence $\text{\rm pr}(K_n = t \mid X_{1:n}) \to 1$ as $n\to\infty$ in $P^{(\infty)}$-probability holds true if and only if one has, in $P^{(\infty)}$-probability,
\begin{equation}\label{condition}
\sum_{s \neq t} \frac{\text{\rm pr}(K_n = s \mid X_{1:n})}{\text{\rm pr}(K_n = t \mid X_{1:n})} \to 0 \quad \text{as } n\to\infty\,.
\end{equation}
\end{Lem}
Working with the ratios of conditional probabilities in \eqref{condition} is beneficial, as the marginal distribution of $X_{1:n}$ involved in the definition of $\text{pr}(K_n = t \mid X_{1:n})$ cancels. Also, it is convenient to write such ratios of probabilities as follows:
 first, recall from \eqref{eq:EPPFDP} and \eqref{first_step} that
\[
\text{pr}( X_{1:n}= x_{1:n}, K_n = s ) 
= \int \dfrac{\alpha^s}{\alpha^{(n)}} \pi(\alpha) \mathrm{d}\alpha \sum_{A \in \tau_s(n)} \prod_{j=1}^s (a_j-1)!m(x_{A_j})
\]
for every $s\geq 1$, which implies that
\begin{equation}\label{newratios}
\frac{\text{pr}(K_n = s \mid X_{1:n})}{\text{pr}(K_n = t \mid X_{1:n})}  = \underbrace{\frac{\int \dfrac{\alpha^s}{\alpha^{(n)}} \pi(\alpha) \, \mathrm{d}\alpha}{\int \dfrac{\alpha^t}{\alpha^{(n)}} \pi(\alpha) \, \mathrm{d}\alpha}}_{C(n,t, s)}\underbrace{\frac{\sum_{A \in \tau_s(n)}\prod_{j=1}^s(a_j-1)!\prod_{j=1}^s m(X_{A_j})}{\sum_{B \in \tau_t(n)}\prod_{j=1}^t(b_j-1)!\prod_{j=1}^t m(X_{B_j})}}_{R(n,t,s)}.
\end{equation}
The decomposition of \eqref{newratios} into the factors $C(n,t,s)$ and $R(n,t,s)$ is useful to understand the role of the prior distribution over $\alpha$, and to compare our results with the one of \citet{Miller2013,Miller2014}. In particular, the term $R(n,t,s)$ does not depend on $\alpha$ and, hence, on the choice of %prior 
$\pi$. %\textcolor{red}{and remains unchanged even if $\alpha$ is fixed}. 
This is indeed the key term studied in \citet{Miller2014}, where it is shown that, under some assumptions, $\liminf R(n,t,s)>0$ as $n\to\infty$ in $P^{(\infty)}$-probability, for $t < s$. On the contrary, %the 
$C(n,t,s)$ incorporates information about $\alpha$ and its prior distribution. 
In the fixed $\alpha$ case, which can be thought of as having a degenerate prior $\pi=\delta_{\alpha}$ for some $\alpha>0$, the term $C(n,t,s)$ boils down to $\alpha^{s-t}$ %$\alpha^{s-1}$,
which is constant with respect to $n$. This is sufficient for \citet{Miller2014} to deduce lack of consistency for fixed $\alpha$, which means that
\begin{equation}\label{eq:inconst_fixed_alpha}
\lim \sup \, \text{pr}( K_n = t \mid X_{1:n},\alpha) <1
\end{equation}
as $n \to \infty$ in $P^{(\infty)}$-probability for every $\alpha>0$.

However, once a non-degenerate prior $\pi$ is employed, %the term 
$C(n,t,s)$ depends on $n$ and, as we show in the next section, converges to 0 as $n\to\infty$ under mild assumptions on $\pi$. Thus, $\liminf R(n,t,s)>0$ is not everymore sufficient to establish whether consistency holds true or not.
Instead, one needs to compare the rate at which $C(n,t,s)$ converges to 0 with the behaviour of $R(n,t,s)$, as done in the following sections.
Note that further lower bounds for $R(n,t,s)$ for general values of $s$ are given in \cite{Miller2014,Yang2019}.
However, once combined with $C(n,t,s)$, these are too loose %rough 
to deduce either consistency or lack thereof. Therefore, we need to exploit different techniques to determine the rate of $R(n,t,s)$.
Since $\text{pr}( K_n = t \mid X_{1:n}) = \int \text{pr}( K_n = t \mid X_{1:n}, \alpha) \pi(\alpha \mid X_{1:n}) \, \mathrm{d}\alpha$, by \eqref{eq:inconst_fixed_alpha} we deduce $\lim \sup \, \text{pr}( K_n = t \mid X_{1:n},\alpha) <1$ for every $\alpha>0$. 
This, however, does not imply that $\lim\sup\text{pr}( K_n = t \mid X_{1:n})<1$, as one first needs to ascertain whether limit and integral can be interchanged. %because in this case the limit cannot be brought inside the integral.
The main reason is that, in the asymptotic regime we are considering, the posterior distribution $\pi(\alpha \mid X_{1:n})$ concentrates around $0$ as $n\to\infty$, see Proposition \ref{prop:PostAlpha} above. 
\subsection{Asymptotic behaviour of the concentration parameter}\label{sec:alpha}
We are now concerned with studying $C(n,t,s)$ in \eqref{newratios}. 
We prove that for priors $\pi$ satisfying assumptions $A1$--$A3$ % in section \ref{sec:consistency_results} 
$C(n,t,s)$ converges to $0$ at a logarithmic rate in $n$.
The asymptotic behaviour of $C(n,t,s)$ is not specific to some  %unrelated to the choice of 
kernel $k$ and data generating distribution $f$ and thus can be useful to prove consistency, or lack thereof, for arbitrary Dirichlet process mixture models with random concentration parameter.
In order to facilitate the intuition, the term $C(n,t,s)$ can be interpreted as a moment of $\alpha$, conditional on the $n$ observations being clustered in $t$ groups. Indeed, under \eqref{eq:MixDPM} it holds
\[
\pi(\alpha \mid K_n=t) \propto \frac{\alpha^t}{\alpha^{(n)}} \pi(\alpha)
\]
and thus $C(n,t,t+s) = \int \alpha^{s} \pi(\alpha \mid K_n=t) \, d\alpha = E(\alpha^{s} \mid K_n = t)$. Next proposition shows its asymptotic behaviour.
\begin{Prop}\label{prop:Prior}
Suppose $\pi$ satisfies %assumptions 
$A1$--$A2$. Then there exist $F, G >0$ such that for every $0 < s \leq n-t$ %it holds
\[
\begin{aligned}
F \, \frac{\gamma\{t+s+\beta,\epsilon\log(n)\}}{\{\log(n)+1\}^{s}} \leq C(n,t, t+s) \leq \frac{Gs}{\epsilon^s}E(\alpha^{t+s-1})\frac{\gamma\{t+s+\beta,\epsilon\log(n)\}}{\{\log n/(1+\epsilon)\}^{s}},
\end{aligned}
\]
where $\gamma(x,y)$ is the lower incomplete Gamma function and %(\textcolor{blue}{non \`e gi\`a chiaro che il momento \`e calcolato rispeto alla prior?})
%\textcolor{red}{
$E(\alpha^s) = \int \alpha^s \pi(\alpha) \, d\alpha$.
\end{Prop}
Thus, for a fixed $s$ that does not depend on $n$, $C(n,t,t+s)$ decreases logarithmically as a function of $n$ since $\gamma(x,y) \leq \gamma(x)$ for every $x$ and $y$. Thus, by looking at the ratios in \eqref{newratios}, the addition of a prior favours a smaller number of clusters when $n \to \infty$, with $s$ fixed.

The consistency results of the previous section are established by combining Proposition~\ref{prop:Prior} with suitable upper bounds on $R(n,t,s)$ to prove the convergence in \eqref{condition}, so that
\[
E \left\{ \sum_{s = 1}^{n-t} \frac{\text{\rm pr}(K_n = t+s \mid X_{1:n})}{\text{\rm pr}(K_n = t \mid X_{1:n})} \right\} \leq \frac{1}{\log n} \sum_{s = 1}^{n-t}h(s),
\]
where $h(s)$ is a %deterministic 
function that depends on the specific kernel $k$ and is such that $\lim \sup$ $\sum_{s = 1}^{n}h(s) < \infty$ for every $s$. Indeed, instead of proving directly convergence in probability of \eqref{condition}, we show %will prove 
the stronger $L^1$ convergence. % in, that is a sufficient condition. 
In this way we will avoid the study of the specific partition at hand. The following lemma shows how the problem simplifies in this case, when $t = 1$.
\begin{Lem}\label{suff_expectation}
Assume $(X_1,X_2,\dots)$ is an exchangeable sequence. Then for every $n$
\begin{equation*}
	E\left\{\sum_{A \in \tau_s(n)}\frac{\prod_{j=1}^s(a_j-1)!}{(n-1)!}\frac{\prod_{j=1}^s m(X_{A_j})}{m(X_{1:n})} \right\} = \sum_{\bm{a} \in \mathcal{F}_s(n)} \frac{n}{s! \prod_{j=1}^{s} a_j}  E\left\{ \frac{\prod_{j=1}^s m(X_{A^{\bm{a}}_j})}{m(X_{1:n})} \right\},
\end{equation*}
where the sum runs over $\mathcal{F}_s(n) = \{ \bm{a} \in \{1,\ldots,n\}^s: \sum_{j = 1}^s a_j=n\}$ and $A^{\bm{a}}$ is an arbitrary partition in $\tau_s(n)$ such that $|A_j^{\bm{a}}| = a_j$ for $j=1,\dots,s$. 
\end{Lem}
\section{Discussion}\label{discussion}
There are mevery avenues to extend our results and some of the tools we introduced here may prove useful to accomplish such tasks. First of all, 
the separability assumption given in Theorem \ref{th:ConsBounded} could be relaxed to prove consistency in the setting with a general number of components. The main issue is that $R(n,t,s)$ in \eqref{newratios} is harder to study, since it becomes the ratio of sums over the space of partitions: in particular Lemma \ref{suff_expectation} is not easy to generalize and this explains why the case $t = 1$ is simpler to address. 
Different mixture kernels present similar difficulties, since they require to study $R(n,t,s)$ for each specific case. Summarising, the impact of the prior is fully understood, by Proposition \ref{prop:Prior} above, but a more general positive result would require finer bounds on the likelihood component than the ones available here and in the literature.

Another interesting question worth studying is whether consistency can also be attained by estimating the concentration parameter through maximization of the marginal likelihood, in an empirical Bayes fashion \citep{Liu1996,McAuliffe2006}.
In this paper we preferred to focus on the fully Bayesian approach because it is arguably the one  %approach 
most commonly employed by practitioners using Dirichlet process mixtures.
Moreover, the empirical Bayes estimator of $\alpha$ may not be well defined on $(0,\infty)$ because the marginal likelihood can easily have its maximum at both $0$ or infinity, thus raising theoretical and practical issues.

It is also worth noticing that our consistency results require the kernel to be perfectly specified: even a small amount of misspecification will probably lead the number of clusters to diverge. Indeed, recovering the true density will require an increasing number of components. This phenomenon has been formally studied in \cite{Cai2020} for finite mixture models, when a prior on the number of components is placed.

We note that the asymptotic analysis of the posterior distribution of the number of clusters for Dirichlet process mixtures has recently attracted considerable theoretical interest \citep{Yang2019,Ohn2020,Cai2020}, and has motivated various methodological developments \citep{Miller2017,Zeng2020}. \cite{Ohn2020} showed that, if $\alpha$ is sent deterministically to 0 at appropriate rates as $n\to\infty$, the posterior distribution of the number of clusters concentrates on finite values when data are generated from a finite mixture, which is a necessary condition for consistency.  Such results are similar in spirit to ours, although we consider the substantially different setting where $\alpha$ is learned through a prior, which is arguably more natural in a Bayesian framework. 
Finally, our results also provide an answer, at least partially, to the question of \cite{Yang2019}: \emph{``there exists a natural way to correct the problem instead of truncating the number of clusters?''}, by showing that placing a prior on $\alpha$ can be sufficient to recover consistency. 

\section*{Supplementary material}
Supplementary material includes all the proofs of the theoretical results.

\bibliographystyle{natbib}
\bibliography{DPM_biblio}

\newpage

\rhead{\bfseries\thepage}
\lhead{\bfseries SUPPLEMENTARY MATERIALS}

\baselineskip 25pt
\begin{center}
{\LARGE{Supplementary Materials for\\} 
\bf  Clustering consistency with\\ Dirichlet process mixtures
}
\end{center}

\setcounter{equation}{0}
\setcounter{page}{1}
\setcounter{table}{1}
\setcounter{figure}{0}
\setcounter{Thm}{0}
\setcounter{Lem}{0}
\setcounter{Cor}{0}
\setcounter{Prop}{0}
\setcounter{section}{0}
\numberwithin{table}{section}
\renewcommand{\theequation}{S.\arabic{equation}}
\renewcommand{\thesubsection}{S.\arabic{section}.\arabic{subsection}}
\renewcommand{\theThm}{S.\arabic{Thm}}
\renewcommand{\theCor}{S.\arabic{Cor}}
\renewcommand{\theProp}{S.\arabic{Prop}}
\renewcommand{\theLem}{S.\arabic{Lem}}
\renewcommand{\thesection}{S.\arabic{section}}
\renewcommand{\thepage}{S.\arabic{page}}
\renewcommand{\thetable}{S.\arabic{table}}
\renewcommand{\thefigure}{S.\arabic{figure}}

\vspace{0cm}

%\iffalse
	\small
\baselineskip=14pt
	\begin{center}
		Filippo Ascolani$^{a,b}$ (filippo.ascolani@phd.unibocconi.it)\\
		Antonio Lijoi$^{a,b}$ (antonio.lijoi@unibocconi.it)\\
		Giovanni Rebaudo$^{b,c}$ (giovanni.rebaudo@austin.utexas.edu)\\
		Giacomo Zanella$^{a,b}$ (giacomo.zanella@unibocconi.it)
	
	%%%%%%%%%%%%%%%%%%%%%%%%%%%%%
	\vskip 3mm
	$^{a}$Department of Decision Sciences, Bocconi University,\\
	via R\"oentgen 1, 20136 Milan, Italy
	\vskip 4pt 
	$^{b}$Bocconi Institute for Data Science and Analytics, Bocconi University,\\ 
	Via Röntgen 1, 20136 Milan, Italy
	\vskip 4pt 
	$^{c}$Department of Statistics and Data Sciences,
	University of Texas at Austin,\\
	105 East 24th Street D9800, Austin, TX 78712, USA\\	
\end{center}

\vskip 10mm
\baselineskip=\gnat

\section*{Summary}
This document contains the proofs of all the results in the main document. In order to follow the logical lines of the arguments, we present the proofs of Sections \ref{sec:Methodology} and \ref{sec:consistency_results}, in this order.

\section{Proof of Lemma \ref{sufficient}}
\begin{proof}
The result immediately follows upon noting that
\[
\text{pr}(K_n = t \mid X_{1:n})=  \left\{1+\sum_{s \neq t} \frac{\text{pr}(K_n = s \mid X_{1:n} )}{\text{pr}(K_n = t \mid X_{1:n})}\right\}^{-1}.
\]
\end{proof}

\section{Proof of Proposition \ref{prop:Prior}}
By assumptions $A1$ and $A2$ there exist $\epsilon, \delta, \beta > 0$ such that
\begin{equation}\label{motivation}
\frac{1}{\delta^2}\frac{\int_0^\epsilon \frac{\alpha^{t+s+\beta}}{\alpha^{(n)}}\, \mathrm{d}\alpha}{\int_0^\epsilon \frac{\alpha^{t+\beta}}{\alpha^{(n)}} \, \mathrm{d}\alpha}  \leq
\frac{\int_0^\epsilon \frac{\alpha^{t+s}}{\alpha^{(n)}}\pi(\alpha)\, \mathrm{d}\alpha}{\int_0^\epsilon \frac{\alpha^t}{\alpha^{(n)}}\pi(\alpha) \, \mathrm{d}\alpha}  \leq \delta^2\frac{\int_0^\epsilon \frac{\alpha^{t+s+\beta}}{\alpha^{(n)}}\, \mathrm{d}\alpha}{\int_0^\epsilon \frac{\alpha^{t+\beta}}{\alpha^{(n)}} \, \mathrm{d}\alpha}\,.
\end{equation}
Notice that, if assumption $A2$ holds for $\epsilon \geq 1$, it holds also for $\epsilon < 1$. Thus, without loss of generality, we will assume $\epsilon < 1$ and the main object of interest will be
\[
%\frac{\int_0^\epsilon \frac{\alpha^{s+1}}{\alpha^{(n)}} \, \pi(\alpha)\mathrm{d}\alpha}{\int_0^\epsilon \frac{\alpha}{\alpha^{(n)}} \, \pi(\alpha)\mathrm{d}\alpha} 
%%= 
%%\int_0^\epsilon \frac{ \frac{\alpha^{s+1}}{\alpha^{(n)}} \, \pi(\alpha)\mathrm{d}\alpha}{\int_0^\epsilon \frac{x}{x^{(n)}}\pi(x) \, \mathrm{d}x} 
%=
E_{n}(\alpha^{s})= \int_0^\epsilon \alpha^{s} p_{n}(\alpha) \mathrm{d}\alpha,
\]
where $E_n$ denotes the expected value with respect to the probability distribution with density
\begin{equation}\label{n_density}
p_n(\alpha) = \frac{f_n(\alpha)}{\int_0^\epsilon f_n(x)\, \mathrm{d}x}, \quad 
f_n(x) = \frac{x^{t+\beta}}{x^{(n)}}\:\mathbbm{1}_{(0,\epsilon)}(x),
\end{equation}
where $\mathbbm{1}_A$ stands for the indicator function of set $A$. We now provide three lemmas that will be useful to prove Proposition 1.

\begin{Lem}\label{lemma:stoch_ord}
Let $f$ and $g$ be two pdf's on $\R$ such that $g(x)/f(x)$ is non-decreasing in $x$.
Then $\int h(x)f(x)\mathrm{d}x\leq \int h(x)g(x)\mathrm{d}x$ for every non-decreasing $h:\R\to\R$.
\end{Lem}
\begin{proof}
Let $X\sim f$ and $Y\sim g$.
Since $g(x)/f(x)$ is non-decreasing we have $g(x_0)f(x_1) \le g(x_1)f(x_0)$ for every $x_0 < x_1$. 
Thus we have %by monotonicity of integral we get
\begin{align*}
F_{Y}(x_1) f(x_1) &= \int_{-\infty}^{x_1} g(x_0)f(x_1) \mathrm{d}x_0 \le  \int_{-\infty}^{x_1} g(x_1)f(x_0) \mathrm{d}x_0 = F_{X}(x_1) g(x_1)
\end{align*}
and
\begin{align*}
%F_{Y}(x_1) f(x_1) &= \int_{-\infty}^{x_1} g(x_0)f(x_1) \mathrm{d}x_0 \le  \int_{-\infty}^{x_1} g(x_1)f(x_0) \mathrm{d}x_0 = F_{X}(x_1) g(x_1),\\
\{1-F_{X}(x_0)\} g(x_0) &= \int_{x_0}^{\infty} g(x_0)f(x_1) \mathrm{d}x_1 \le  \int_{x_0}^{\infty} g(x_1)f(x_0) \mathrm{d}x_1 = \{1-F_{Y}(x_0) \} f(x_0).
\end{align*}
It follows
%\begin{align*}
%    \frac{F_{Y}(x)}{F_{X}(x)} \le \frac{g(x)}{f(x)}, \quad \frac{1-F_{Y}(x)}{1-F_{X}(x)} \textcolor{red}{\ge} \frac{g(x)}{f(x)}.
%\end{align*}
\begin{align*}
\frac{F_{Y}(x)}{F_{X}(x)} \le \frac{g(x)}{f(x)}\le\frac{1-F_{Y}(x)}{1-F_{X}(x)}\,,% \textcolor{red}{\ge} \frac{g(x)}{f(x)}.
\end{align*}
for every $x\in \R$,
which implies
\begin{align*}
\frac{F_{Y}(x)}{1-F_{Y}(x)} \le \frac{F_{X}(x)}{1-F_{X}(x)}\,.% \textcolor{red}{\ge} \frac{g(x)}{f(x)}.
\end{align*}
%The last two inequalities imply that $F_Y(x)\{1-F_X(x)\} \le F_X(x) \{1-F_Y(x) \}$. 
Thus, $Y$ stochastically dominates $X$, i.e.\ the corresponding cdf's satisfy $F_Y(x)~\leq~F_X(x)$ for every $x\in \R$, which implies that $E\{h(X)\}\leq E\{h(Y)\}$ for every non-decreasing $h$.
\end{proof}

\begin{Lem}\label{main}
Under assumptions $A1$ and $A2$, for every $n-t > s \geq 1$ it holds
\[
\frac{\gamma[t+s+\beta,\epsilon\{\log(n)+1\}]}{\delta^2\gamma[t+\beta,\epsilon\{\log(n)+1\}]}\{\log(n)+1\}^{-s}
\leq
\frac{\int_0^\epsilon \frac{\alpha^{t+s}}{\alpha^{(n)}}\pi(\alpha)\, \mathrm{d}\alpha}{\int_0^\epsilon \frac{\alpha^t}{\alpha^{(n)}}\pi(\alpha) \, \mathrm{d}\alpha} 
\leq
\frac{\delta^2 \gamma\{t+s+\beta,\epsilon\log(n)\}}{\gamma\{t+\beta,\epsilon\log(n)\}}\{\log(n)/(1+\epsilon)\}^{-s},
\]
where $\gamma(x,y)$ is the lower incomplete Gamma function and we recall that $\epsilon, \delta, \beta >0$ are such that for every $ \alpha \in (0, \epsilon)$ it holds $\frac{1}{\delta}\alpha^\beta \leq \pi(\alpha) \leq \delta \alpha^\beta$.
\end{Lem}
\begin{proof}
By \eqref{motivation} it suffices to find suitable bounds of $E_n(\alpha^{s})$.
For the upper inequality we apply Lemma \ref{lemma:stoch_ord} with $f=p_n$, $g(\alpha)\propto (cn)^{-\alpha}\alpha^{t+\beta-1}\mathbbm{1}_{\left(\alpha\in[0,\epsilon]\right)}$ with $c=(1+\epsilon)^{-1}$ and $h(\alpha)=\alpha^s$.
To verify that $g(\alpha)/p_n(\alpha)$ is non-decreasing for $\alpha\in(0,\epsilon]$ we compute 
\begin{align*}
\frac{\mathrm{d}}{\mathrm{d}\alpha}\log\left\{\frac{g(\alpha)}{p_n(\alpha)}\right\}=&
-\log \left(\frac{n}{1+\epsilon}\right)+\sum_{i=1}^{n-1}\frac{1}{\alpha+i}\\
\geq&
-\log \left(\frac{n+\epsilon}{1+\epsilon}\right)+\sum_{i=1}^{n-1}\frac{1}{i+\epsilon}
\geq 0,
\end{align*}
where the last inequality follows from 
\[
\int_1^{k}\frac{1}{x+\epsilon} \, dx < \sum_{i = 1}^{k-1}\frac{1}{i+\epsilon}
\]
for every $k > 1$. Thus, since $h(\alpha)=\alpha^s$ is non-decreasing in $\alpha$ it follows by Lemma \ref{lemma:stoch_ord} that
$$
\begin{aligned}
E_n(\alpha^{s})
\leq&
\frac{\int_0^\epsilon \alpha^{t+s+\beta-1}(cn)^{-\alpha} \mathrm{d}\alpha}{\int_0^\epsilon \alpha^{t+\beta-1}(cn)^{-\alpha}\, \mathrm{d}\alpha} 
=
\frac{\{\log(cn)\}^{-s}\int_0^{\epsilon\log(cn)} z^{t+s+\beta-1}e^{-z} \mathrm{d}z}{\int_0^{\epsilon\log(cn)} z^{t+\beta-1}e^{-z}\, \mathrm{d}z} 
\\
=&
\frac{\{\log(cn)\}^{-s}\gamma\{t+s+\beta,\epsilon\log(cn)\}}{\gamma\{t+\beta,\epsilon\log(cn)\}}.
\end{aligned}
$$
The lower bound again follows from Lemma \ref{lemma:stoch_ord} with $f(\alpha) \propto (e n)^{-\alpha}\alpha^{t+\beta-1}\mathbbm{1}_{\left(\alpha\in[0,\epsilon]\right)}$, $g(\alpha)=p_n(\alpha)$ and $h(\alpha)=\alpha^s$.
To verify that $p_n(\alpha)/f(\alpha)$ is non-decreasing for $\alpha\in(0,\epsilon]$ we compute
\begin{align*}
\frac{\mathrm{d}}{\mathrm{d}\alpha}\log\left\{\frac{p_n(\alpha)}{f(\alpha)}\right\}=&
-\sum_{i=1}^{n-1}\frac{1}{\alpha+i}+\log (n)+1\\
\geq&
-\sum_{i=1}^{n-1}\frac{1}{i}+\log (n)+1
\geq 0,
\end{align*}
where the last inequality follows from
\[
\sum_{i = 1}^{k}\frac{1}{i} \leq \log(k)+1
\]
for every $k \geq 1$. Thus, since $h(\alpha)=\alpha^s$ is non-decreasing in $\alpha$, we have 
$$
\begin{aligned}
E_n(\alpha^{s})
\geq&
\frac{\int_0^\epsilon \alpha^{t+s+\beta-1}(e n)^{-\alpha} \mathrm{d}\alpha}{\int_0^\epsilon \alpha^{t+\beta-1}(e n)^{-\alpha}\, \mathrm{d}\alpha} 
=
\frac{\{\log(en)\}^{-s}\int_0^{\epsilon\log(en)} z^{t+s+\beta-1}e^{-z} \mathrm{d}z}{\int_0^{\epsilon\log(en)} z^{t+\beta-1}e^{-z}\, \mathrm{d}z} \\
&=
\frac{\{\log(en)\}^{-s}\gamma\{t+s+\beta,\epsilon\log(en)\}}{\gamma\{t+\beta,\epsilon\log(en)\}}.
\end{aligned}
$$
The proof is completed by combining the bounds with \eqref{motivation}.
\end{proof}

\begin{Lem}\label{lower_tech}
For every $\epsilon >0$, there exists $M > 0$ such that, for every $n \geq 1$, it holds
\[
M \,\int_0^\epsilon \frac{\alpha^t}{\alpha^{(n)}}\, \pi(\alpha) \, \mathrm{d}\alpha \geq \int_\epsilon^\infty \frac{\alpha^t}{\alpha^{(n)}}\, \pi(\alpha) \, \mathrm{d}\alpha\,.
\]
\end{Lem}
\begin{proof}
Define $p = \frac{\int_\epsilon^\infty \alpha^t \pi(\alpha) \,\mathrm{d}\alpha}{\int_0^{\frac{\epsilon}{2}} \alpha^t \pi(\alpha) \,\mathrm{d}\alpha}$. Then
\begin{align*}
\int_0^\epsilon \frac{\alpha^t}{\alpha^{(n)}}\, \pi(\alpha) \, \mathrm{d}\alpha - \int_\epsilon^\infty \frac{\alpha^t}{\alpha^{(n)}}\, \pi(\alpha) \, \mathrm{d}\alpha=&
\int_0^{\epsilon} \frac{\alpha^t}{\alpha^{(n)}}\, \pi(\alpha) \, \mathrm{d}\alpha - \int_0^{\frac{\epsilon}{2}} p\frac{\alpha^t}{\epsilon^{(n)}}\, \pi(\alpha) \, \mathrm{d}\alpha\\
\geq& \int_0^{\frac{\epsilon}{2}} \frac{\alpha^t}{\alpha^{(n)}}\, \pi(\alpha) \, \mathrm{d}\alpha - \int_0^{\frac{\epsilon}{2}} p\frac{\alpha^t}{\epsilon^{(n)}}\, \pi(\alpha) \, \mathrm{d}\alpha.
\end{align*}
Choose $m$ such that $\left(\frac{\epsilon}{2}\right)^{(m)} < \frac{\epsilon^{(m)}}{p}$, which is always possible because $\left\{\epsilon^{(m)}\right\}^{-1} \left(\frac{\epsilon}{2}\right)^{(m)}\to 0$ as $m\to\infty$. Thus
\[
\int_0^\epsilon \frac{\alpha^t}{\alpha^{(n)}}\, \pi(\alpha) \, \mathrm{d}\alpha \geq \int_\epsilon^\infty \frac{\alpha^t}{\alpha^{(n)}}\, \pi(\alpha) \, \mathrm{d}\alpha, \quad n \geq m
\]
and it suffices to set $M = \max (P,1)$ with
\[
P = \max_{1 \leq i \leq m} \left\{ \frac{\int_\epsilon^\infty \frac{\alpha^t}{\alpha^{(i)}}\, \pi(\alpha) \, \mathrm{d}\alpha}{\int_0^\epsilon \frac{\alpha^t}{\alpha^{(i)}}\, \pi(\alpha) \, \mathrm{d}\alpha} \right\}\,.
\]
\end{proof}

\begin{proof}[Proof of Proposition \ref{prop:Prior}]
We first prove the upper bound. We have%On the one side
\begin{align*}
C(n,t,t+s) &\leq \frac{\int_0^\infty \frac{\alpha^{t+s}}{\alpha^{(n)}}\pi(\alpha)\, \mathrm{d}\alpha}{\int_0^\epsilon \frac{\alpha^t}{\alpha^{(n)}}\pi(\alpha) \, \mathrm{d}\alpha}
= \frac{\int_0^\epsilon \frac{\alpha^{t+s}}{\alpha^{(n)}}\pi(\alpha)\, \mathrm{d}\alpha}{\int_0^\epsilon \frac{\alpha^t}{\alpha^{(n)}}\pi(\alpha) \, \mathrm{d}\alpha}+ \frac{\int_0^\epsilon \frac{\alpha^{t+s}}{\alpha^{(n)}}\pi(\alpha)\, \mathrm{d}\alpha}{\int_0^\epsilon \frac{\alpha^t}{\alpha^{(n)}}\pi(\alpha) \, \mathrm{d}\alpha} \frac{\int_\epsilon^\infty \frac{\alpha^{t+s}}{\alpha^{(n)}}\pi(\alpha)\, \mathrm{d}\alpha}{\int_0^\epsilon \frac{\alpha^{t+s}}{\alpha^{(n)}}\pi(\alpha) \, \mathrm{d}\alpha}.
\end{align*}
Moreover, %from assumption $A2$ 
it holds 
\[
\frac{\int_\epsilon^\infty \frac{\alpha^{t+s}}{\alpha^{(n)}}\pi(\alpha)\, \mathrm{d}\alpha}{\int_0^\epsilon \frac{\alpha^{t+s}}{\alpha^{(n)}}\pi(\alpha) \, \mathrm{d}\alpha} \leq \frac{\int_\epsilon^\infty \alpha^{t+s-1}\pi(\alpha)\, \mathrm{d}\alpha}{\int_0^\epsilon \alpha^{t+s-1}\pi(\alpha) \, \mathrm{d}\alpha} \leq \delta\frac{\int_\epsilon^\infty \alpha^{t+s-1}\pi(\alpha)\, \mathrm{d}\alpha}{\int_0^\epsilon \alpha^{t+s+\beta-1}\, \mathrm{d}\alpha} \leq \delta \, E(\alpha^{t+s-1}) \frac{t+s+\beta}{\epsilon^{t+s+\beta}},
\]
where the first inequality follows since $\alpha^{(n)} \geq \epsilon^{(n)}$ for $\alpha \in (\epsilon, \infty)$ and $\alpha^{(n)} \leq \epsilon^{(n)}$ for $\alpha \in (0,\epsilon)$, while the second one follows from assumption $A2$. 
Moreover, $E$ stands for the expected value with respect to $\pi$.
Thus from Lemma \ref{main} it holds 
\[
C(n,t,t+s) \leq \frac{\delta^2 \left\{1+E(\alpha^{t+s-1})\, \frac{t+s+\beta}{\epsilon^{t+s+\beta}}  \right\}\gamma\{t+s+\beta,\epsilon\log(n)\}}{\gamma\{t+\beta,\epsilon\log(n)\}}\{\log(n)/(1+\epsilon)\}^{-s}.
\]
Then choose $G = \frac{4\delta^2}{\epsilon^{t+\beta}\gamma(t+\beta, \epsilon \log 2)}$ to obtain the upper bound. 
For the lower bound, apply Lemma \ref{main} and Lemma \ref{lower_tech} to get
\[
C(n,t,t+s) \geq \frac{1}{M+1}\frac{\int_0^\epsilon \frac{\alpha^{t+s}}{\alpha^{(n)}}\pi(\alpha)\, \mathrm{d}\alpha}{\int_0^\epsilon \frac{\alpha^t}{\alpha^{(n)}}\pi(\alpha) \, \mathrm{d}\alpha}  \geq \frac{1}{M+1}\frac{\gamma[t+s+\beta,\epsilon\{\log(n)+1\}]}{\delta^2\gamma[t+\beta,\epsilon\{\log(n)+1\}]}\{\log(n)+1\}^{-s}.
\]
Then choose $F = \frac{1}{(M+1)\delta^2 \gamma(t+\beta)}$.
\end{proof}
The following corollary of Proposition \ref{prop:Prior} will be useful.
\begin{Cor}\label{cor_C}
Suppose $\pi$ satisfies assumptions $A1$ and $A2$. Then $G >0$ as in Proposition \ref{prop:Prior} is such that for every $0 < s < n$ and $n \geq 4$ it holds
\[
\begin{aligned}
C(n,t,t+s) \leq \frac{G\Gamma(t+\beta+1)2^ss}{\epsilon}E(\alpha^{t+s-1})\log\{n/(1+\epsilon)\}^{-1}.
\end{aligned}
\]
\end{Cor}
\begin{proof}
By Proposition \ref{prop:Prior} we have
\[
C(n,t,t+s) \leq \frac{Gs}{\epsilon^s}E(\alpha^{t+s-1})\frac{\gamma\{t+s+\beta,\epsilon\log(n)\}}{\log\{n/(1+\epsilon)\}^{s}}.
\]
Note that
\[
\gamma\{t+s+\beta,\epsilon\log(n)\} = \int_0^{\epsilon \log (n)} x^{t+s+\beta-1}e^{-x}  \, \d x \leq \epsilon^{s-1}\{\log (n) \}^{s-1}\Gamma(t+\beta+1),
\]
that implies
\[
\frac{\gamma\{t+s+\beta,\epsilon\log(n)\}}{\epsilon^s \log^s\{n/(1+\epsilon)\}} \leq \frac{\Gamma(t+\beta+1)}{\epsilon}\left[\frac{\log(n)}{\log\{n/(1+\epsilon)\}} \right]^{s-1}\log\{n/(1+\epsilon)\}^{-1}.
\]
Moreover, since $\epsilon < 1$, we have $\log\{n/(1+\epsilon)\} \geq \frac{1}{2}\log(n)$ for every $n \geq 4$.
Combining the inequalities above we obtain the desired result.
\end{proof}

\section{Proof of Lemma \ref{suff_expectation}}
\begin{proof}
We need to study $R(n,1,s)$ as in \eqref{newratios}. Taking the expectation with respect to the data generating distribution we have
\begin{align*}
E\{R(n,1,s)\} &= \sum_{A \in \tau_s(n)} \frac{\prod_{j=1}^s (a_j-1)! }{(n-1)! } E\left\{\frac{\prod_{j=1}^s m(X_{A_j})}{m(X_{1:n})}\right\}\\
&=	\sum_{\bm{a} \in \mathcal{F}_s(n)} \binom{n}{a_1 \cdots a_j} \frac{\prod_{j=1}^s (a_j-1)! }{s! (n-1)! } E\left\{\frac{\prod_{j=1}^s m(X_{A_j^{\bm{a}}})}{m(X_{1:n})}\right\}\\
&=\sum_{\bm{a} \in \mathcal{F}_s(n)} \frac{n}{s! \prod_{j=1}^{s} a_j}  E\left\{\frac{\prod_{j=1}^s m(X_{A_j^{\bm{a}}})}{m(X_{1:n})}\right\}.
\end{align*}
\end{proof}

\section{Proof of Lemma \ref{priors}}
\begin{proof}
Assumptions $A1$ and $A2$ are immediately satisfied in all three cases discussed in the statement of the lemma. We thus focus on proving that $A3$ is satisfied, considering each of the three cases separately.
Suppose first that the support of the density $\pi$ is contained in $[0,c]$ with $c > 0$. Then
\[
\int_0^{\infty} \alpha^s \pi(\alpha) \, \d\alpha \leq c^s.
\]
Thus in this case assumption $A3$ is satisfied for every $\rho>0$ because $c^s < D\rho^{-s}\Gamma(s+1)$ with $D = \underset{s \in \mathbb{N}}{\max} \frac{(c\rho)^s}{\Gamma(s+1)}$ for every $\rho > 0$.
Suppose now the prior is given by a Generalized Gamma distribution, so that
\[
\int_0^{\infty} \alpha^s \pi(\alpha) \, \d\alpha = \frac{p}{a^d\Gamma\left(\frac{d}{p} \right)} \int_0^{\infty} \alpha^{d+s-1}e^{-\left(\frac{\alpha}{a} \right)^p} \, \d\alpha\,.
\]
The condition $p>1$ implies that, for every fixed $\rho > 0$ and $a > 0$, there exists $k > 0$ such that $\rho \alpha \leq \left(\frac{\alpha}{a}\right)^p$ for every $\alpha \geq k$. Thus
\[
\begin{aligned}
\int_0^{\infty} \alpha^{d+s-1}e^{-\left(\frac{\alpha}{a} \right)^p} \, \d\alpha &\leq \int_0^k \alpha^{s+d-1}e^{-\left(\frac{\alpha}{a}\right)^p} \, \d \alpha + \int_k^\infty \alpha^{s+d-1}e^{-\rho \alpha} \, \d \alpha \\
& \leq  k^{s+d-1}e^{-\left(\frac{k}{a}\right)^p} + \rho^{-d-s}\Gamma(s+d).
\end{aligned}
\]
Also,
\[
\begin{aligned}
\int_0^{\infty} \alpha^s \pi(\alpha) \, \d\alpha \leq& \frac{p}{a^d\Gamma\left(\frac{d}{p} \right)}\Gamma(s+d) \left\{  \frac{k^{s+d-1}e^{-\left(\frac{k}{a}\right)^p}}{\Gamma(s+d)} + \rho^{-d-s}\right\} \leq\\
& \leq D\rho^{-s}\Gamma(s+d),
\end{aligned}
\]
with $D = \underset{s \in \mathbb{N}}{\max} \, \frac{p}{a^d\Gamma\left(\frac{d}{p} \right)} \left\{  \frac{k^{s+d-1}e^{-\left(\frac{k}{a}\right)^p\rho^s}}{\Gamma(s+d)} + \rho^{-d}\right\}$, so that also in this case assumption $A3$ is satisfied for every $\rho > 0$. 
Finally, in the case of Gamma distribution we get
\[
\int_0^{\infty} \alpha^s \pi(\alpha) \, \d\alpha = \frac{\Gamma(\nu+s)}{\Gamma(\nu)}\rho^{-s}
\]
and assumption $A3$ holds.
\end{proof}

\section{Proof of Theorem \ref{th:ConsBounded}}

Through a linear rescaling, we may assume $[a, b] = [-c, c]$ without loss of generality.  We rewrite the assumptions on $g$ and $Q_0$ as
\begin{enumerate}
\item[$T1.$]  $\exists \, m, M$ such that $0 < m \leq g(x) \leq M < \infty$ for every $x \in [-c, c]$;
\item[$T2.$]  $g$ is differentiable on $(-c, c)$ and $\exists \, R$ such that $|\frac{g'(x)}{g(x)}| \leq R < \infty$ for every $x \in (-c, c)$;
\item[$T3.$]  $\exists \, U > 0$ such that $h(y) = q_0(y)+q_0(-y) \leq U$ for every $y \in [0, 2c]$;
\item[$T4.$]  $\exists \, L > 0$ such that $q_0(\theta) \geq L$ for every $\theta$ in a neighborhood of $\theta_j^*$, for every $j$.
\end{enumerate}
Denote with $f(x) = \sum_{j = 1}^tp_jk(x \mid \theta_j^*)$ the density of the data generating $P = \sum_{j = 1}^tp_jR_j$, with $t \in \mathbb{N}$, $p_j \in (0,1)$ and $\sum_{j = 1}^tp_j = 1$. Since $\theta^* =(\theta_1^*, \dots, \theta_t^*)$ is completely separated and \\$X^{\infty} \sim P^{(\infty)}$, each point $x$ has non-null density for at most one component of the mixture, i.e.
\[
x \in [\theta_i^*+a, \theta_i^*+b] \quad \Rightarrow \quad f(x) = p_ik(x \mid \theta_i^*) = p_ig(x-\theta_i^*).
\]
Therefore we can define
\[
C_j = \left\{ i \in \{1, \dots, n\} \, : \, x_i \in [\theta_j^*+a, \theta_j^*+b] \right\}, \quad n_j = |C_j|.
\]
Notice that $C_i \cap C_j = \emptyset$ for every $i \neq j$ and $\{1, \dots, n \} = \bigcup_{j = 1}^tC_j$, so that $\sum_{j = 1}^t n_j = n$. Moreover, defining
\[
C^{(n)} = \left\{ n_j > 0 \text{ for every $j$} \right\},
\]
for every $x_{1:n} \in C^{(n)}$ it holds
\begin{equation}\label{less_and_equal}
\begin{aligned}
&\sum_{A \in \tau_s(n)}\prod_{j=1}^s(a_j-1)!\prod_{j=1}^s m(x_{A_j}) = 0 \quad \text{for every } s < t,\\
&\sum_{B \in \tau_t(n)}\prod_{j=1}^t(b_j-1)!\prod_{j=1}^t m(x_{B_j}) = \prod_{j = 1}^t(n_j-1)!\prod_{j = 1}^tm(x_{C_j}).
\end{aligned}
\end{equation}
Since $p_j > 0$ for every $j = 1, \dots, s$, we have $P^{(n)}(C^{(n)}) \to 1$ as $n \to \infty$. We need a technical lemma.
\begin{Lem}\label{lemma:ConvProbBound}
Let $\Omega_n$ be a sequence of sets depending on $X_{1:n}$, and let $Z_n$ be random variables on the same probability space such that $P^{(\infty)}(\Omega_n) \to 1$ and
\begin{equation*}
Z_n \mathbbm{1}_{\Omega_n} \to 0
\end{equation*}
in $P^{(\infty)}$-probability as $n \to \infty$. Then $Z_n \to 0$ in $P^{(\infty)}$-probability as $n \to \infty$.
\end{Lem}
\begin{proof}
By assumption $P^{(\infty)}\left(\mathbbm{1}_{\Omega_n} Z_n > \epsilon \right) \to 0$ as $n \to \infty$. Thus, we have
\[
P^{(\infty)}\left(Z_n > \epsilon \right) \leq P^{(\infty)}\left\{(Z_n > \epsilon) \cap \Omega_n \right\}+ P^{(\infty)}\left(\Omega_n^c\right) \to 0
\]
as $n \to \infty$.
\end{proof}
Thus by Lemma \ref{lemma:ConvProbBound} it suffices to study
\begin{equation}\label{ratios_with_all_groups}
\frac{\text{pr}(K_n = s \mid X_{1:n})}{\text{pr}(K_n = t \mid X_{1:n})}\mathbbm{1}_{C^{(n)}}  = \frac{\int \dfrac{\alpha^s}{\alpha^{(n)}} \pi(\alpha) \, \mathrm{d}\alpha}{\int \dfrac{\alpha^t}{\alpha^{(n)}} \pi(\alpha) \, \mathrm{d}\alpha}\frac{\sum_{A \in \tau_s(n)}\prod_{j=1}^s(a_j-1)!\prod_{j=1}^s m(X_{A_j})}{\sum_{B \in \tau_t(n)}\prod_{j=1}^t(b_j-1)!\prod_{j=1}^t m(X_{B_j})}\mathbbm{1}_{C^{(n)}}.
\end{equation}
By \eqref{less_and_equal}, we have
\[
\frac{\text{pr}(K_n = s \mid X_{1:n})}{\text{pr}(K_n = t \mid X_{1:n})}\mathbbm{1}_{C^{(n)}}  = 0
\]
for every $s < t$. Let us now consider the case $s > t$. Again by complete separability, $A \in \tau_s(n)$ yields positive marginal density only if $A$ is a refinement of the partition $\left\{ C_1, \dots, C_t\right\}$, i.e. if
\[
A \in \tilde{\tau}_s(n) = \left\{A \in \tau_s(n) \, : \, \forall \, i = 1, \dots, s \text{ there exists } j \in \{1, \dots, t\} \text{ such that } A_i \subset C_j  \right\}.
\]
Therefore, if $A \in \tilde{\tau}_s(n)$, we write the $j$-the element as $A_j = (A^j_1, \dots, A^j_{s_j})$ with $a^j_k = |A^j_k|$, so that
\[
\sum_{A \in \tilde{\tau}_s(n)}\prod_{j=1}^s(a_j-1)!\prod_{j=1}^s m(X_{A_j}) = \sum_{\textbf{s}\in \text{S}}\prod_{j = 1}^t\sum_{A_j \in \tau_{s_j}(n_j)}\prod_{k=1}^{s_j}(a^j_k-1)!\prod_{k=1}^{s_j} m(X_{A^j_k}),
\]
where $\textbf{S} = \left\{(s_1, \dots, s_t) \, : \, 1 \leq s_j \leq n_j, \, \forall j, \text{ and } \sum_{j = 1}^ts_j = s \right\}$. By the above and \eqref{less_and_equal} we can rewrite \eqref{ratios_with_all_groups} as
\begin{equation}\label{ratios_for_inconsistency}
\begin{aligned}
\frac{\text{pr}(K_n = s \mid X_{1:n})}{\text{pr}(K_n = t \mid X_{1:n})}\mathbbm{1}_{C^{(n)}} & = C(n, t, s)\frac{\sum_{A \in \tilde{\tau}_s(n)}\prod_{j=1}^s(a_j-1)!\prod_{j=1}^s m(X_{A_j})}{\prod_{j = 1}^t(n_j-1)!\prod_{j = 1}^tm(X_{C_j})}\mathbbm{1}_{C^{(n)}} \\
&= C(n, t, s)\sum_{\textbf{s}}\prod_{j = 1}^t\sum_{A_j \in \tau_{s_j}(n_j)}\frac{\prod_{k=1}^{s_j}(a^j_k-1)!}{(n_j-1)!}\frac{\prod_{k=1}^{s_j} m(X_{A^j_k})}{m(A_{C_j})}\mathbbm{1}_{C^{(n)}},
\end{aligned}
\end{equation}
where
\[
m(X_{C_j}) = \int_{\mathbb{R}}\prod_{i \in C_j}k(X_i \mid \theta_j) \, Q_0(\d \theta_j) = \int_{\mathbb{R}}\prod_{i \in C_j}g(X_i - \theta_j) \, Q_0(\d \theta_j)
\]
and
\[
m(X_{A^j_h}) = \int_{\mathbb{R}}\prod_{i \in A^j_h}k(X_i \mid \theta_h) \, Q_0(\d \theta_h) = \int_{\mathbb{R}}\prod_{i \in A^j_h}g(X_i - \theta_h) \, Q_0(\d \theta_h),
\]
with $h = 1, \dots, s_j$. We divide and multiply by
\[
\prod_{i = 1}^nf(X_i) = \prod_{j = 1}^t\prod_{i \in C_j}p_jk(X_i \mid \theta_j^*) = \prod_{j = 1}^t\prod_{h = 1}^{s_j}\prod_{i \in A_h^j}p_jk(X_i \mid \theta_j^*),
\]
so that the sum on the right hand side of \eqref{ratios_for_inconsistency} becomes
\begin{equation}\label{object}
\sum_{\textbf{s}}\prod_{j = 1}^t\sum_{A_j \in \tau_{s_j}(n_j)}\frac{\prod_{k=1}^{s_j}(a^j_k-1)!}{(n_j-1)!}\frac{\prod_{k = 1}^{s_j}\int_{\mathbb{R}}\prod_{i \in A^j_k}\frac{g(X_i - \theta_k)}{p_jg(X_i-\theta_j^*)} \, Q_0(\d \theta_k)}{\int_{\mathbb{R}}\prod_{i \in C_j}\frac{g(X_i - \theta_j)}{p_jg(X_i-\theta_j^*)} \, Q_0(\d \theta_j)}\mathbbm{1}_{C^{(n)}}, \quad \text{for } s > t.
\end{equation}
We start with the denominator. The next lemma specifies the behaviour of the maximum for each group, where $X^j_{(r)}$ denotes the $r$-th order statistic of $X_{C_j}$.
\begin{Lem}\label{lemma:Max}
For every $j = 1, \dots , t$ it holds
\[
Y^j_{n_j} := \min \left[ 1, n_j(\log(n))^{\frac{1}{2t}}\{c+\theta_j^*-X^j_{(n_j)}\} \right] \to 1
\]
in $P^{(\infty)}$-probability as $n \to \infty$.
\end{Lem}
\begin{proof}
First, notice that $n_j \to \infty$ $P^{(\infty)}$-almost surely as $n \to \infty$. By definition $Y^j_{n_j} \leq 1$, so we have to prove that $\forall \epsilon > 0$
\[
P^{(\infty)}\left(1- Y^j_{n_j} > \epsilon \right) \to 0
\]
as $n_j \to \infty$, where pr is evaluated with respect to $P^{(\infty)}$. Without loss of generality assume $\theta_j^* = 0$. Thus, by definition we have
\[
\begin{aligned}
P^{(\infty)}(1 - Y^j_{n_j} > \epsilon) &=P^{(\infty)}\left[n_j(\log(n))^{\frac{1}{2t}}\{c-X^j_{(n)} \} \leq 1-\epsilon \right] = P^{(\infty)}\left\{ X^j_{(n)} \geq c-\frac{1-\epsilon}{n_j(\log(n))^{\frac{1}{2t}}}\right\} \\
&= 1-\left\{1-\int_{c-\frac{1-\epsilon}{n_j(\log(n))^{\frac{1}{2t}}}}^c g(x) \, \mathrm{d}x \right\}^{n}\,.
\end{aligned}
\]
Thus, by $T1$ we have that $\int_{c-\frac{1-\epsilon}{n_j(\log(n))^{\frac{1}{2t}}}}^c g(x) \, \mathrm{d}x \leq \frac{M(1-\epsilon)}{n_j(\log(n))^{\frac{1}{2t}}}$, so that
\[
P^{(\infty)}(1 - Y^j_{n_j} > \epsilon) \leq 1-\left\{ 1- \frac{M(1-\epsilon)}{n_j(\log(n))^{\frac{1}{2t}}}\right\}^{n}= 1-e^{-\frac{M(1-\epsilon)}{(\log(n))^{\frac{1}{2t}}}+n_j \, \text{o}\left(\frac{1}{n_j(\log(n))^{\frac{1}{2t}}}\right)} \to 0,
\]
as $n \to \infty$, by the Taylor expansion of the logarithmic function.
\end{proof}
\begin{Lem}\label{denom1}
For every $j = 1, \dots, t$ it holds
\[
\prod_{i \in C_j} \frac{g(x_i - \theta_j)}{g(x_i)} \geq e^{-R} \mathbbm{1}_{[0,\frac{1}{n_j}]} (|\theta_j-\theta_j^*|) \mathbbm{1}_{[x^j_{(n_j)}-c,x^j_{(1)}+c]}(\theta_j-\theta_j^*).
\] 
with $R$ defined in $T2$ and $x^j_{(r)}$ denotes the $r$-th order statistic of $x_{C_j}$.
\end{Lem}
\begin{proof}
Without loss of generality assume $\theta_j^* = 0$. Define $p(x) := \log g(x)$, with $x \in [-c, c]$, so that $p'(x) = \frac{g'(x)}{g(x)}$. By $T2$ and the Fundamental Theorem of Integral Calculus
\[
|p(y)-p(x)| = \left\lvert \int_x^yp'(t) \, \d t \right\rvert \leq \int_x^y \left|\frac{g'(t)}{g(t)}\right| \d t \leq R|y-x|, \quad -c < x \leq y < c.
\]
Thus, we have 
\[
\frac{g(x - \theta_j)}{g(x)} = e^{p(x -\theta_j)-p(x)} = e^{-\{p(x)-p(x - \theta_j)\}} \geq e^{-R|\theta_j|}, \quad x \in [-c,c].
\]
Finally, we get
\[
\begin{aligned}
\prod_{i \in C_j} \frac{g(x_i - \theta_j)}{g(x_i)} &\geq e^{-Rn_j|\theta_j|}\mathbbm{1}_{[x^j_{(n_j)}-c,x^j_{(1)}+c]}(\theta_j) \geq e^{-Rn|\theta_j|}\mathbbm{1}_{[0,\frac{1}{n_j}]} (|\theta_j|)\mathbbm{1}_{[x^j_{(n_j)}-c,x^j_{(1)}+c]}(\theta_j)\\
& \geq e^{-R} \mathbbm{1}_{[0,\frac{1}{n_j}]} (|\theta_j|) \mathbbm{1}_{[x^j_{(n_j)}-c,x^j_{(1)}+c]}(\theta_j).
\end{aligned}
\]
\end{proof}
\begin{Lem}\label{denom}
For every $j = 1, \dots, t$ there exists $K>0$ and $N_j \in \mathbb{N}$ such that for all $n_j\geq N_j$  it holds
\[
\int_\R \prod_{i \in C_j}\frac{g(X_i - \theta_j)}{g(X_i-\theta_j^*)}q_0(\theta_j) \, \mathrm{d}\theta_j %&\geq Le^{-2c}\min \left\{\frac{1}{n}, c-X_{(n)} \right\} = \frac{Le^{-2c}\min \left\{1, n(c-X_{(n)}) \right\}}{n} \geq \\
\geq %\frac{c \min \left[1, n\sqrt{\log(n)}\{c-X_{(n)}\} \right]}{n\sqrt{\log(n)}} = 
\frac{K^{\frac{1}{t}} Y^j_{n_j} }{n_j(\log(n))^{\frac{1}{2t}}},
\]
with $Y^j_{n_j}$ defined in Lemma \ref{lemma:Max}.
\end{Lem}
\begin{proof}
Without loss of generality assume $\theta_j^* = 0$. Notice that, by $T4$, there exists $N_j \in \mathbb{N}$ such that $q_0(\theta) \geq L$ for every $\theta \in \left[-\frac{1}{N_j},0\right]$. Thus, applying Lemma \ref{denom1} and considering $n_j \geq N_j$, we get
\[
\begin{aligned}
\int_\R \prod_{i \in C_j}\frac{g(X_i - \theta_j)}{g(X_i)}q_0(\theta_j) \, \mathrm{d}\theta_j&\geq e^{-R}\int_\R \mathbbm{1}_{[0,\frac{1}{n_j}]} (|\theta_j|) \mathbbm{1}_{[X^j_{(n_j)}-c,x^j_{(1)}+c]}(\theta_j) \, q_0(\theta_j)\, \mathrm{d}\theta_j\\
%& = e^{-2c}\int_{-\frac{1}{n}}^{\frac{1}{n}}\mathbbm{1}_{[x_{(n)}-c,x_{(1)}+c]}(\theta) \, q_0(\theta) \, \mathrm{d}\theta \geq \\
& \geq e^{-R}\int_{-\frac{1}{n_j}}^0\mathbbm{1}_{\{X^j_{(n_j)} \leq \theta_j+c\}} \, q_0(\theta_j) \, \mathrm{d}\theta_j \geq Le^{-R}\min \left\{\frac{1}{n_j}, c-X^j_{(n_j)} \right\},
\end{aligned}
\]
with $L$ defined in $T4$. Thus, multiplying both the numerator and the denominator by $n_j(\log(n))^{\frac{1}{2t}}$, with $n \geq N$, we have
\[
\begin{aligned}
\int_\R \prod_{i \in C_j}\frac{g(X_i - \theta_j)}{g(X_i)}q_0(\theta_j) \, \mathrm{d}\theta_j &\geq 2 Le^{-R}\min \left\{\frac{1}{n_j}, c-X^j_{(n_j)} \right\} \\%= \frac{c \min \left\{1, n(c-X_{(n)}) \right\}}{n} \geq \\
& \geq \frac{K^{\frac{1}{t}} \min \left[1, n_j(\log(n))^{\frac{1}{2t}}\{c-X_{(n)}\} \right]}{n_j(\log(n))^{\frac{1}{2t}}} =  \frac{K^{\frac{1}{2t}} Y_n }{n_j(\log(n))^{\frac{1}{2t}}},
\end{aligned}
\]
with $K = (2Le^{-R})^t$.
\end{proof}
Define the event
\begin{equation}\label{Omega_n}
\Omega_n = \left\{\text{for every $j = 1, \dots, t$ it holds: } n_j \geq N_j, Y^j_{n_j} \in [1/2, 1] \right\},
\end{equation}
such that $P^{(n)}(\Omega_n) \to 1$ thanks to Lemma \ref{lemma:Max} and Lemma \ref{denom}. Thus, an upper bound of \eqref{object} with $\Omega_n$ in place of $C^{(n)}$ is given by
\begin{equation}\label{upper_bound_ratios}
T^{(n)} := \frac{2^t\sqrt{\log(n)}}{K}\sum_{\textbf{s}}\prod_{j = 1}^t\sum_{A_j \in \tau_{s_j}(n_j)}n_j\frac{\prod_{k=1}^{s_j}(a^j_k-1)!}{(n_j-1)!}\prod_{h = 1}^{s_j}\int_{\mathbb{R}}\prod_{i \in A^j_h}\frac{g(X_i - \theta_h)}{g(X_i-\theta_j^*)} \, Q_0(\d \theta_h)\mathbbm{1}_{\Omega_n},
\end{equation}
for $s > t$. Now we apply the expected value with respect to the values of each group, as shown in the next lemma.

\begin{Lem}\label{numerator}
Under $X_{1:n} \sim P^{(n)}$, for every $j = 1, \dots, t$, $s_j \geq 1$ and $(\theta_1, \dots, \theta_{s_j}) \in \mathbb{R}^{s_j}$, we have
\[
\begin{aligned}
E \left\{\prod_{h = 1}^{s_j}\int_{\R^{s_j}} \prod_{i \in A^j_h} \frac{g(X_i - \theta_h)}{g(X_i-\theta_j^*)} q_0(\theta_h)\, \mathrm{d}\theta_h \right\} 
&\leq \left(\frac{U}{m}\right)^{s_j}\prod_{h= 1}^{s_j} \frac{1}{a^j_h+1},
\end{aligned}
\]
with $m$ and $U$ defined in $T1$ and $T3$. 
\end{Lem}
\begin{proof}
Without loss of generality assume $\theta_j^* = 0$. Taking the expectation under $P^{(n)}$ we have
\begin{equation}\label{expectation_num}
\begin{aligned}
E \left\{\int_{\R^{s_j}} \prod_{h = 1}^{s_j} \prod_{i \in A^j_h} \frac{g(X_i - \theta_h)}{g(X_i)} q_0(\theta_h)\, \mathrm{d}\theta_h \right\} %&= \int_{[-c, c]^n}\int_{\Theta^s} \prod_{j = 1}^s \prod_{i \in A_j} k(x_i - \theta_j) q_0(\theta_j) \, \mathrm{d}\theta_j\, \mathrm{d}x_i\\
& = \int_{\R^{s_j}} \int_{[-c,c]^{n_j}} \prod_{h = 1}^{s_j} \prod_{i \in A^j_h} g(x_i - \theta_h) q_0(\theta_h) \, \mathrm{d}x_i \, \mathrm{d}\theta_h,
\end{aligned}
\end{equation}
By the change of variables $z = x-\theta_h$, we have
\[
\int_{-c}^{c}  g(x - \theta_h) \mathbbm{1}_{[\theta_h-c, \theta_h+c]}(x)\, \mathrm{d}x = \int_{-c-\theta_h}^{c-\theta_h} g(z)\mathbbm{1}_{[-c, c]}(z)\, \mathrm{d}z.
\]
If $\theta_h > 0$, then
\[
\begin{aligned}
\int_{-c-\theta_h}^{c-\theta_h} g(z)\mathbbm{1}_{[-c, c]}(z)\, \mathrm{d}z &= \mathbbm{1}_{[0, 2c]}(\theta_h)\int_{-c}^{c-\theta_h}g(z) \, \mathrm{d}z\\
& = \mathbbm{1}_{[0, 2c]}(\theta_h)\left(1-\int_{c-\theta_h}^c g(z) \, \mathrm{d}z\right) \leq \mathbbm{1}_{[0, 2c]}(|\theta_h|)\left(1-m|\theta_h|\right).
\end{aligned}
\]
Similarly, if $\theta_h < 0$ we get
\[
\begin{aligned}
\int_{-c-\theta_h}^{c-\theta_h} g(z)\mathbbm{1}_{[-c, c]}(z)\, \mathrm{d}z &= \mathbbm{1}_{[-2c, 0]}(\theta_h)\int_{-c-\theta_h}^{c}g(z) \, \mathrm{d}z\\
&= \mathbbm{1}_{[-2c, 0]}(\theta_h)\left( 1-\int_{-c}^{-c-\theta_h} g(z) \, \mathrm{d}z\right) \leq \mathbbm{1}_{[0,2c]}(|\theta_h|) \left(1-m|\theta_h|\right).
\end{aligned}
\]
Thus
\[
\int_{-c}^{c}  g(x - \theta_h) \mathbbm{1}_{[\theta_h-c, \theta_h+c]}(x)\, \mathrm{d}x  \leq \mathbbm{1}_{[0,2c]}(|\theta_h|) \left(1-m|\theta_h|\right), \quad h = 1, \dots, s_j,
\]
which implies
\[
\prod_{h = 1}^{s_j} \prod_{i \in A^j_h} \int_{-c}^{c}  g(x - \theta_h) \mathbbm{1}_{[\theta_h-c, \theta_h+c]}(x)\, \mathrm{d}x  \leq \prod_{h = 1}^{s_j}\mathbbm{1}_{[0,2c]}(|\theta_h|) \left(1-m|\theta_h|\right).
\]
Considering $h$ defined as in $T3$, we have
\[
\int_{\R}\mathbbm{1}_{[0,2c]}(|\theta_h|) \left(1-m|\theta_h|\right)q_0(\theta_h) \, \d \theta_h = \int_0^{2c} \left(1-m|\theta_h|\right)h(\theta_h) \, \d \theta_h, \quad h =1, \dots, s_j.
\]
Combining the above with \eqref{expectation_num} we get
\begin{equation}\label{expectation_num2}
\begin{aligned}
E \left\{\int_{\R^{s_j}} \prod_{h = 1}^{s_j} \prod_{i \in A^j_h} \frac{g(X_i - \theta_h)}{g(X_i)} q_0(\theta_h)\, \mathrm{d}\theta_h \right\} %&= \int_{[-c, c]^n}\int_{\Theta^s} \prod_{j = 1}^s \prod_{i \in A_j} k(x_i - \theta_j) q_0(\theta_j) \, \mathrm{d}\theta_j\, \mathrm{d}x_i\\
& = \int_{\R^{s_j}} \int_{[-c,c]^{n_j}} \prod_{h = 1}^{s_j} \prod_{i \in A^j_h} g(x_i - \theta_h) q_0(\theta_h) \, \mathrm{d}x_i \, \mathrm{d}\theta_h\\
&\leq \prod_{h=1}^{s_j}\int_0^{2c} \left(1-m|\theta_h|\right)h(\theta_h) \, \d \theta_h.
\end{aligned}
\end{equation}
With $U$  defined as in $T3$, we have
\[
\int_{0}^{2c} (1-my)^{a_h^j} h(y) \, \mathrm{d}y \leq U \int_{0}^{2c} (1-m y)^{a_h^j} \, \mathrm{d}y.
\]
Now consider the change of variables $u = 1-my$ and compute
\[
\begin{aligned}
\int_{0}^{2c} (1-my)^{a_h^j} \, \mathrm{d}y &= \frac{1}{m} \int_{1-2 m c}^1u^{a_h^j} \, \mathrm{d}u = \frac{1-(1-2mc)^{a_h^j+1}}{m(a_h^j+1)} \leq \frac{1}{m(a_h^j+1)}.
\end{aligned}
\]
Finally, through \eqref{expectation_num2}, we have
\[
\begin{aligned}
E \left\{\int_{\R^{s_j}} \prod_{h = 1}^{s_j} \prod_{i \in A^j_h} \frac{g(X_i - \theta_h)}{g(X_i)} q_0(\theta_h)\, \mathrm{d}\theta_h \right\}  %&= \int_{[-c, c]^n}\int_{\Theta^s} \prod_{j = 1}^s \prod_{i \in A_j} k(x_i - \theta_j) q_0(\theta_j) \, \mathrm{d}\theta_j\, \mathrm{d}x_i\\
& \leq \prod_{h=1}^{s_j}\int_0^{2c} \left(1-m|\theta_h|\right)h(\theta_h) \, \d \theta_h\\
&\leq \left(\frac{U}{m}\right)^{s_j}\prod_{h = 1}^{s_j} \frac{1}{a_h^j+1},
\end{aligned}
\]
as desired.
\end{proof}

\subsection{Proof of Theorem \ref{th:ConsBounded}}
We have the next two technical lemmas.
\begin{Lem}\label{bound_sum}
Let $p^* = \min_{j \in \{1, \dots, t\}} p_j \in (0,1)$. It holds
\[
\sum_{\textbf{s} \in \textbf{S}}\frac{s!}{\prod_{j = 1}^ts_j!} = \sum_{\textbf{s}}\binom{s}{s_1, \dots, s_t} \leq (p^*)^{-s},
\]
where $\textbf{S} = \left\{(s_1, \dots, s_t) \, : \, s_j \leq n_j \text{ and } \sum_{j = 1}^ts_j = s \right\}$.
\end{Lem}
\begin{proof}
The result follows immediately from
\[
\begin{aligned}
\sum_{\textbf{s} \in \textbf{S}}\binom{s}{s_1, \dots, s_t} &\leq (p^*)^{-s}\sum_{\textbf{s} \in \textbf{S}}\binom{s}{s_1, \dots, s_t}\prod_{j = 1}^tp_j^{s_j}\\
& \leq (p^*)^{-s}\sum_{\textbf{s} \in R_t}\binom{s}{s_1, \dots, s_t}\prod_{j = 1}^tp_j^{s_j}= (p^*)^{-s},
\end{aligned}
\]
where $R_t = \left\{(s_1, \dots, s_t) \, : \, \sum_{j = 1}^ts_j = s \right\}$, since the sum on the right-hand side is the sum of the probabilities over all the possible values of a multinomial distribution with parameters $(s, p_1, \dots, p_t)$.
\end{proof}
\begin{Lem}\label{induction}
For every $p>1$ and for every integers $s\geq 2$ and $n\ge s$ it holds
\begin{align*}
\sum_{\bm{a} \in \mathcal{F}_s(n)} \left(\frac{n}{\prod_{j=1}^s a_j}\right)^p < C_p^{s-1},
\end{align*}
where $\mathcal{F}_s(n) = \left\{ \textbf{a} \in \{1,\ldots,n\}^s: \sum_{j = 1}^s a_j =n\right\}$
and  $C_p = 2^{p}\zeta(p)$, with $\zeta(p) = \sum_{a=1}^{\infty} \frac{1}{a^p}<\infty$.
\end{Lem}
\begin{proof}
We prove the result by induction. Consider the base case $s=2$. 
By the strict convexity of $x\mapsto x^p$ for $p>1$ we have
\begin{align*}
\sum_{\bm{a} \in \mathcal{F}_2(n)} \left(\frac{n}{a_1 a_2}\right)^p = \sum_{a=1}^{n-1} \left\{\frac{n}{a(n-a)}\right\}^p
= 2^p \sum_{a=1}^{n-1} \left(\frac{1}{2}\frac{1}{a} +\frac{1}{2} \frac{1}{n-a}\right)^p < 2^{p} \sum_{a=1}^{n-1} \frac{1}{a^p} < C_p,
\end{align*}
for every $n\geq 2$. 
For the induction step, assume that for some $s\geq 3$ we have
$$
\sum_{\bm{a}\in \mathcal{F}_{s-1}(n)} \left(\frac{n}{\prod_{j=1}^{s-1} a_j}\right)^2 < C_p^{s-2}
$$
for all $n\geq s-1$. Then
\begin{align*}
\sum_{\bm{a} \in \mathcal{F}_s(n)} \left(\frac{n}{\prod_{j=1}^s a_j}\right)^p 
&=
\sum_{a_{s}=1}^{n-s+1} \sum_{(a_1,\dots,a_{s-1})\in \mathcal{F}_{s-1}(n-a_{s})} \left(\frac{n}{\prod_{j=1}^s a_j}\right)^p 
\\&=
\sum_{a_{s}=1}^{n-s+1} \left\{\frac{n}{(n-a_{s})a_{s}}\right\}^p \sum_{(a_1,\dots,a_{s-1})\in \mathcal{F}_{s-1}(n-a_{s})} \left(\frac{n-a_s}{\prod_{j=1}^{s-1} a_j}\right)^p\\
&\le C_p^{s-2} \sum_{a_{s}=1}^{n-s+1} \left\{\frac{n}{(n-a_{s})a_{s}}\right\}^p  < C_p^{s-1}
\end{align*}
%meaning that the desired statement holds for $s$. T
and thus the thesis follows by induction.
\end{proof}
In the following we will drop the subscript in $C_p$ when the value of $p$ is clear from the context, thus denoting $C = C_p$.
\begin{Lem}\label{lemma:consGeneral}
Consider the setting of \eqref{eq:MixDPM} with $(f, k, q_0)$ as in Theorem \ref{th:ConsBounded}. Moreover, assume $\pi(\alpha)$ satisfies assumptions $A1$, $A2$, and $A3$. Then, under $X_{1:\infty} \sim P^{(\infty)}$ we have
\[
E\left\{\mathbbm{1}_{\Omega_n} \sum_{s = 1}^{n-t} \frac{\text{pr}(K_{n} = t+s \, | \, X_{1:n})}{\text{pr}(K_{n} = t \, | \, X_{1:n})} \right\} \to 0
\]
as $n \to \infty$, with $\Omega_n$ as in \eqref{Omega_n}.
\end{Lem}
\begin{proof}
Applying Lemma \ref{numerator} we can upper bound the expected value of $T^{(n)}$ in \eqref{upper_bound_ratios} as follows
\[
\begin{aligned}
\mathbb{E}\left\{T^{(n)}\right\} \leq \frac{2^t\sqrt{\log(n)}}{K}\left(\frac{U}{m}\right)^s&\sum_{\textbf{s}}\prod_{j = 1}^t\sum_{A_j \in \tau_{s_j}(n_j)}\frac{n_j}{(n_j-1)!\prod_{k = 1}^{s_j}(a^j_k+1)}\\
& \leq \frac{2^t\sqrt{\log(n)}}{K}\left(\frac{U}{m}\right)^s\sum_{\textbf{s}}\prod_{j = 1}^t\frac{1}{s_j!}\sum_{\textbf{a}_j \in \mathcal{F}_{s_j}(n_j)}\left(\frac{n_j}{\prod_{k = 1}^{s_j}a^j_k}\right)^2,
\end{aligned}
\]
where the last inequality follows from Lemma \ref{suff_expectation}. Moreover, from Lemma \ref{induction} we have
\[
\sum_{\textbf{a}_j \in \mathcal{F}_{s_j}(n_j)}\left(\frac{n_j}{\prod_{k = 1}^{s_j}a^j_k}\right)^2 < C^{s_j},
\]
with constant $C < 7$. Thus
\begin{equation}\label{bound_T}
\mathbb{E}\left\{T^{(n)}\right\} \leq \frac{2^t\sqrt{\log(n)}}{K}\left(\frac{UC}{m}\right)^s\sum_{\textbf{s}}\prod_{j = 1}^t\frac{1}{s_j!}.
\end{equation}
Moreover, from Corollary \ref{cor_C} and $A3$ we have
\begin{equation}\label{bound_C}
\begin{aligned}
C(n,t,t+s) &\leq \frac{G\Gamma(t+\beta+1)2^ss}{\epsilon}E(\alpha^{t+s-})\log\{n/(1+\epsilon)\}^{-1} \\
& \leq \frac{DG\Gamma(t+\beta+1)2^ss}{\epsilon}\rho^{-(t+s-1)}\Gamma(\nu+t+s)\log\{n/(1+\epsilon)\}^{-1}, \quad n \geq 4\,.
\end{aligned}
\end{equation}
By \eqref{bound_T}, combined with Lemma \ref{bound_sum}, and \eqref{bound_C} we finally have
\[
\begin{aligned}
E \biggl\{\mathbbm{1}_{\Omega_n}&\sum_{s = 1}^{n-t}\frac{\text{pr}(K_n = s+t | X_{1:n})}{\text{pr}(K_n = t | X_{1:n})}\biggr\} = \sum_{s = 1}^{n-t}C(n,t,t+s)E\{\mathbbm{1}_{\Omega_n}R(n,t,t+s)\}\\
&  \leq \frac{2^t\rho^{1-t}(U/m)^tDG\Gamma(t+\beta+1)\sqrt{\log(n)}}{K\epsilon \log\{n/(1+\epsilon)\}}\underbrace{\sum_{ s = 1}^{n-1}\frac{s (2CUp^*/m)^{s}\rho^{-s}\Gamma(\nu +t+s)}{(s+1)!}}_{< \infty} \to 0,
\end{aligned}
\]
as $n \to \infty$, where finiteness follows by taking $\rho$ sufficiently large.
\end{proof}
\begin{proof}[Proof of Theorem \ref{th:ConsBounded}] First of all, assume $\pi(\cdot)$ satisfies $A1-A3$. By Lemma \ref{lemma:consGeneral} it holds
\[
\mathbbm{1}_{\Omega_n}\sum_{s = 1}^{n-t} \frac{\text{pr}(K_{n} = t+s \, | \, X_{1:n})}{\text{pr}(K_{n} = t \, | \, X_{1:n})} \to 0
\]
in $P^{(\infty)}$--probability as $n \to \infty$. The desired result then follows from Lemma \ref{lemma:ConvProbBound} with $Z_n = \sum_{s = 1}^{n-t} \frac{\text{pr}(K_{n} = t+s \, | \, X_{1:n})}{\text{pr}(K_{n} =t \, | \, X_{1:n})}$ and $\Omega_n$ as in \eqref{Omega_n}.

Assume instead $\pi(\alpha) = \delta_{\alpha^*}(\alpha)$ with $\alpha^* > 0$. By \eqref{ratios_for_inconsistency} we have
\[
\frac{p(K_n = t+1 \mid X_{1:n})}{p(K_n = t \mid X_{1:n})} \geq \alpha^*\sum_{\textbf{s} \in \textbf{S}}\prod_{j = 1}^t\sum_{A_j \in \tau_{s_j}(n_j)}\frac{\prod_{k=1}^{s_j}(a^j_k-1)!}{(n_j-1)!}\frac{\prod_{k=1}^{s_j} m(X_{A^j_k})}{m(A_{C_j})}.
\]
Notice that, with $n$ high enough, $n_1 > 1$ almost surely. Then, denoting $i \in C_1$, we consider the special case
\[
\textbf{s} = (2,1, \dots, 1), \quad A^1_1 = \{i\}, A_2^1 = A_{C_1} \backslash \{i\},
\]
and $A_j = \{A_{C_j}\}$ for every $j \geq 2$. Thus we can write
\begin{equation}\label{inequality_inconsistency}
\frac{p(K_n = t+1 \mid X_{1:n})}{p(K_n = t \mid X_{1:n})} \geq \alpha^*\sum_{i \in C_1}\frac{1}{n_1-1}\frac{m(X_i)m\left(X_{C_1 \backslash i} \right)}{m\left(X_{C_j} \right)}.
\end{equation}
By $T1$ we have
\[
\begin{aligned}
m\left(X_{C_j} \right) &= \int_{\R} \prod_{j \in C_1}g(X_j-\theta)q_0(\theta) \, \d \theta\\
& \leq M \int_{\R} \prod_{j \in C_1 \backslash i}g(X_j-\theta)q_0(\theta) \, \d \theta = M \, m\left(X_{C_1 \backslash i} \right).
\end{aligned}
\]
Moreover, by $T4$ there exists $\epsilon > 0$ such that
\[
\begin{aligned}
m(X_i) = \int_{\R} g(X_i -\theta) q_0(\theta) \d \theta &\geq m \int_{\theta_1^*-\epsilon}^{\theta_1^* + \epsilon}q_0(\theta) \d \theta
&\geq 2mL\epsilon.
\end{aligned}
\]
Therefore, \eqref{inequality_inconsistency} becomes
\[
\frac{p(K_n = t+1 \mid X_{1:n})}{p(K_n = t \mid X_{1:n})} \geq \frac{2\alpha^*mL\epsilon}{M}\sum_{i \in C_1}\frac{1}{n_1-1} = \frac{2\alpha^*mL\epsilon}{M}\frac{n_1}{n_1-1},
\]
and
\[
\lim \inf_{n \to \infty} \sum_{s \neq t}\frac{p(K_n = s \mid X_{1:n})}{p(K_n = t \mid X_{1:n})} \geq \lim \inf_{n \to \infty} \frac{p(K_n = t+1 \mid X_{1:n})}{p(K_n = t \mid X_{1:n})} \geq \frac{\alpha^*mL\epsilon}{M} > 0.
\]
Then
\[
\begin{aligned}
\lim \sup_{n \to \infty} \text{pr}(K_n = t \mid X_{1:n}) &= \lim \sup_{n \to \infty} \left\{1+\sum_{s \neq t} \frac{\text{pr}(K_n = s \mid X_{1:n} )}{\text{pr}(K_n = t \mid X_{1:n})}\right\}^{-1}\\
& = \frac{1}{1+\lim \inf_{n \to \infty} \sum_{s \neq t}\frac{p(K_n = s \mid X_{1:n})}{p(K_n = t \mid X_{1:n})}} > 0,
\end{aligned}
\]
which completes the proof.
\end{proof}

\section{Proof of Proposition \ref{undersmoothing}}
We adapt the proof of Theorem $2.1$ in \cite{Cai2020}. Denote by
\[
\Psi = \left\{k(\cdot \mid \theta): \theta \in \Theta \subseteq \mathbb{R}^p \right\}
\]
the family of kernels, dominated by $\mu$, either Lebesgue or counting measure, and with common domain $\mathbb{X} \subseteq \mathbb{R}^q$. Denote with $B_x(\epsilon)$ the closed ball of center $x \in \mathbb{X}$ and radius $\epsilon > 0$. Let $\bar{\Theta}$ be the closure of $\Theta$ and define the set
\[
\mathbb{B} := \left\{ \bar{\theta} \in \bar{\Theta}\backslash \Theta \, : \, \lim_{\theta \to \bar{\theta}} \left \{\sup_x k(x \mid \theta)\right\} = \infty \right\}.
\]
Let $\mathbb{G}_s$ be the set of mixtures of exactly $s$ elements in $\Psi$, that is
\[
f \in \mathbb{G}_s \quad \Leftrightarrow \quad f = \sum_{j = 1}^sq_jk(\cdot \mid \theta_j),
\]
with $q_j > 0$ for every $j$, $\sum_{j = 1}^sq_j = 1$ and $\theta_i \neq \theta_h$ for every $i \neq h$.  Let $\mathcal{P}(G)$ be the set of probability measures on a generic space $G$; with a slight abuse of notation we will say $f \in \mathcal{P}(G)$ when $f$ is the density of a probability measure $P \in \mathcal{P}(G)$. Therefore, given $P \in \mathbb{G}_t$, with weights $\{p_j\}_{j = 1}^t$ and parameters $\{\theta_j^* \}_{j = 1}^t$, we define the Kullback-Leibler neighborhoods of $P$ as
\begin{equation}\label{KL_neighborhood}
KL_\epsilon(P) := \bigg\{ h \in \mathcal{P}(\mathbb{X}): \int  \log\bigg\{ \frac{\sum_{j=1}^{t} p_j k(x\mid \theta_j^\ast)}{h(x)} \bigg\} P(\mathrm{d}x)  < \epsilon \bigg\},
\end{equation}
for $\epsilon > 0$. We make the following assumptions: 
\begin{enumerate}
\item[$H1.$] For every $\bar{\theta} \in \Theta \backslash \mathbb{B}$, for $\mu$-almost every $x \in \mathbb{X}$ there exists $A := A(\bar{\theta}, x) \subset \Theta \backslash \mathbb{B}$ neighborhood of $\bar{\theta}$ so that the mapping $\theta \in A \to k(x \mid \theta)$ is continuous. Moreover $\mathbb{B}$ is closed; 
\item[$H2.$] Let $\{\theta_i \}_{i = 1}^\infty \subset \Theta$. If $||\theta_i|| \to \infty$ as $i \to \infty$, then for every compact set $K \subset \mathbb{X},$ \\$\int_Kk(x\mid \theta_i) \, \mu(\d x) \to 0$, as $i \to \infty$. If $\theta_i \to \bar{\theta} \in \mathbb{B}$, then there exists $x^* \in \mathbb{X}$ such that $k(\cdot \mid \theta_i) \to \delta_{x^*}(\cdot)$ weakly as $i \to \infty$;% for every $\epsilon > 0$ \\$\int_{B_{x^*}(\epsilon)}k(x\mid \theta_i) \, \mu(\d x) \to 1$, as $i \to \infty$;
\item[$H3.$] If $f \in \mathbb{G}_t$, then there exist no $f' \in \mathbb{G}_s$, with $s < t$, such that $f(x) = f'(x)$ $\mu$-almost surely;
\item[$H4.$] For every $P \in \mathbb{G}_t$, $t \geq 1$, with $\theta_1^*, \dots, \theta_t^*$ belonging to the support of $Q_0$, we have \\$\text{pr}(h \in K_\epsilon(P))>0$ for every $\epsilon > 0$, where $h$ follows the prior distribution in \eqref{eq:MixDPM}.
\end{enumerate}
Assumption $H2$ says that, when $\theta$ diverges or converges to elements in $\mathbb{B}$, the kernel $k$ degenerates: it is satisfied for instance when the elements of $\theta$ are location or scale parameters. $H3$ instead implies that the clustering problem is not ill-posed, in the sense that different numbers of components always lead to different distribution. $H4$ finally requires that the finite mixtures of the kernel $k(\cdot \mid \theta)$ belongs to the Kullback-Leibler support of the prior. They are all weak requirements, satisfied by the most common kernels. Next Lemma shows that they are satisfied under assumptions $B1-B3$.
\begin{Lem}\label{th:WeakConsDens}
Suppose the kernel $k(x \mid \theta)$ satisfies assumptions $B1-B3$. Then $H1-H4$ are fulfilled.
\end{Lem}

\begin{proof}
Assumption $H3$ can be easily deduced from $B1$ and \eqref{data_generating_general}. As regards $H1$, since $\sup_{\theta \in \Theta, x \in \mathbb{X}}k(x \mid \theta) < \infty$, we have $\mathbb{B} = \emptyset$. Moreover, fix $\bar{\theta} \in \mathbb{R}$. If $x > \theta +b$, choose
\[
A(\bar{\theta}, x) = \left(\bar{\theta}-\frac{x-\bar{\theta}-b}{2}, \bar{\theta}+\frac{x-\bar{\theta}-b}{2} \right),
\]
so that $x > \theta+b$ that implies $ k(x \mid \theta) = 0$ for every $\theta \in A(\bar{\theta}, x)$. Similarly, if $x < \theta +a$, choose
\[
A(\bar{\theta}, x) = \left(\bar{\theta}-\frac{\bar{\theta}+a-x}{2}, \bar{\theta}+\frac{\bar{\theta}+a-x}{2} \right).
\]
Finally, if $x \in (\bar{\theta}+a, \bar{\theta}+b)$, denoting $d = \min \{\bar{\theta}+b-x, x-\bar{\theta}-a\}$, choose
\[
A(\bar{\theta}, x) = \left(\bar{\theta}-\frac{d}{2}, \bar{\theta}+\frac{d}{2} \right).
\]
Then $k(x \mid \theta) = g(x-\theta)$ for every $\theta \in A(\bar{\theta}, x)$ and $g$ is continuous on $(a,b)$, by $B2$. Thus we can find the required neighborhood $A(\bar{\theta}, x)$ for every $x \not \in \{\bar{\theta}+a, \bar{\theta}+b\}$, that is for $\mu$-almost every $x$, since $\mu$ is the Lebesgue measure. Therefore $H1$ is satisfied.

$H2$ follows since $\theta$ is a location parameter and $\bar{\Theta} = \Theta$. We are left to show that $H4$ is satisfied: we prove the case $t = 1$ and the general setting follows similarly. 

Recall that assumptions $B1-B3$ can be rewritten as $T1-T4$ in the proof of Theorem \ref{th:ConsBounded} and let $f(x) = k(x\mid \theta^*)$ be the density function of $P$. Fix $\delta > 0$, $\epsilon > 0$ and denote $r= 1-exp(\epsilon/4)$. Define the set
\begin{equation}\label{def_F}
\begin{aligned}
\mathbb{F}(\delta, r) := \biggl\{p(x)=\sum_{j=1}^{\infty} q_j k(x \mid \theta_j) \, : \, &q_1 \in [1-r, 1], q_2 \in [r/2, 1],\\
& 0 \leq \theta^\ast -\theta_1 \leq \delta,0 \leq \theta_2 -\theta^\ast \leq \delta \biggr\}.
\end{aligned}
\end{equation}
We denote $[a_j, b_j] := [a+\theta_j, b+\theta_j]$, with $j \geq 1$, and similarly $[a^*, b^*] := [a + \theta^*, b+\theta^*]$. Then we can choose $\delta$ small enough such that
\[
[a_1,b_1] \cup [a_2, b_2] \supseteq [a^\ast, b^\ast],
\]
for every $\theta_1$ and $\theta_2$ as in \eqref{def_F}. Moreover, for every $x \in S_1 := [a_1, b_1] \cap [a^\ast, b^\ast]$ we have
\[
\begin{aligned}
\log\left\{\frac{g(x-\theta^\ast)}{q_1 g(x-\theta_1)} \right\}&=-\log(q_1)+ \log\left\{\frac{g(x-\theta^\ast)}{g(x-\theta_1)}\right\} \leq \epsilon/4 + \log\left\{\frac{g(x-\theta^\ast)}{g(x-\theta_1)}\right\}\\
& \leq \epsilon/4+R|\theta^*-\theta_1|\,
\end{aligned}
\]
with $R > 0$ as in $T2$. Therefore we can choose $\delta$ small enough so that
\begin{equation}\label{second_delta}
\log\left\{\frac{g(x-\theta^\ast)}{q_1 g(x-\theta_1)} \right\} < \frac{\epsilon}{2}
\end{equation}
for every $x \in S_1$. Similarly, we can choose $\delta$ small enough so that for every $x \in S_2 :=[a^\ast, b^\ast] \setminus [a_1, b_1]$ we have
\begin{equation}\label{third_delta}
\int_{S_2 } g(x-\theta^\ast) \log\bigg\{\frac{g(x-\theta^\ast)}{q_{2} g(x-\theta_2)} \bigg\} \mathrm{d}x  < \frac{\epsilon}{2}.
\end{equation}
Indeed, since $g(x-\theta^*) \leq M$ and $m \leq g(x-\theta_2)$ for every $x$ in $S_2$, with $m$ and $M$ as in $T1$, we have
\[
g(x-\theta^\ast) \log\bigg\{\frac{g(x-\theta^\ast)}{q_2 g(x-\theta_2)} \bigg\} < M \log\{2M/(mr) \},
\]
and $S_2$ has arbitrarily small length with $\delta$ small enough. For every $p \in \mathbb{F}(\delta, r)$, by applying \eqref{second_delta} and \eqref{third_delta}, we have
\begin{align*}
&\int_{a^*}^{b^*} g(x-\theta^\ast) \log\bigg\{\frac{g(x-\theta^\ast)}{\sum_{j=1}^{\infty} q_{j} g(x-\theta_{j})} \bigg\}  \mathrm{d}x =\\
&\int_{S_1} g(x-\theta^\ast) \log\bigg\{\frac{g(x-\theta^\ast)}{\sum_{j=1}^{\infty} q_{j} g(x-\theta_{j})} \bigg\}  \mathrm{d}x +
\int_{S_2} g(x-\theta^\ast) \log\bigg\{\frac{g(x-\theta^\ast)}{\sum_{j=1}^{\infty} q_{j} g(x-\theta_{j})} \bigg\}  \mathrm{d}x \le\\
& \int_{S_1} g(x-\theta^\ast) \log\bigg\{\frac{g(x-\theta^\ast)}{ q_{1} g(x-\theta_{1})} \bigg\}  \mathrm{d}x +
\int_{S_2} g(x-\theta^\ast) \log\bigg\{\frac{g(x-\theta^\ast)}{ q_{2} g(x-\theta_{2})} \bigg\}  \mathrm{d}x \le \epsilon.
\end{align*}
Thus, $\mathbb{F}(\delta, r)\subseteq K_\epsilon(P)$ for $\delta$ small enough.
Moreover, since $\theta^*$ belongs to the support of $Q_0$ and the Dirichlet process prior has full weak support on the space of probability weights $\{ q_j \}_j$, we have that
\[
\text{pr}\{h \in K_\epsilon(P)\} \geq \text{pr}\{h \in \mathbb{F}(\delta, r)\} >0,
\]
as desired.
\end{proof}
The proof of Proposition \ref{undersmoothing} will rely on the following Lemma.
\begin{Lem}\label{lemma_weak_convergence}
Let assumption $H4$ be satisfied and let $P \in \mathbb{G}_t$ with parameters $\theta_1^*, \dots, \theta_t^*$ belonging to the support of $Q_0$. Assume there exists $\mathcal{U}$ weak neighborhood of $P$ such that $\mathcal{U} \cap \mathbb{G}_s = \emptyset$ for every $s < t$. Then
\[
\text{pr}\left(K_n < t \mid X_{1:n} \right) \to 0,
\]
in $P^{(\infty)}$-probability as $n \to \infty$.
\end{Lem}
\begin{proof}
By assumption $H4$, the posterior distribution is consistent at $P$ under the weak topology, in virtue of Schwartz theorem (see e.g.Theorem $6.16$ and Example $6.20$ in \cite{Subhashis2017}), so that
\begin{equation}\label{limit}
\text{pr}(h \in \mathcal{U}^c \mid X_{1:n}) \to 0,
\end{equation}
in $P^{(\infty)}$-probability as $n \to \infty$. Moreover, we have
\[
\text{pr}(h \in \mathcal{U}^c \mid X_{1:n}) \geq \text{pr}(h \in \mathcal{U}^c \mid X_{1:n}, K_n < t)\text{pr}\left(K_n < t \mid X_{1:n} \right).
\]
Notice that, conditional on $K_n < t$, the domain of the posterior distribution is a subset of $\cup_{s < t}\mathbb{G}_s$. Thus we have $\text{pr}(h \in \mathcal{U}^c \mid X_{1:n}, K_n < t) = 1$ and
\[
\text{pr}(h \in \mathcal{U}^c \mid X_{1:n}) \geq \text{pr}\left(K_n < t \mid X_{1:n} \right).
\]
The result follows from \eqref{limit}.
\end{proof}
We need two technical Lemmas.

\begin{Lem}\label{first_auxiliary}
Assume a sequence $\{f_i\}_{i = 1}^\infty \subset \cup_{s < t}\mathbb{G}_s$ is such that $f_i \to f \in \mathcal{P}(\mathbb{X})$ weakly as $i \to \infty$. Then there exist $s' < t$ and a sequence $\{ f'_i\}_{i = 1}^\infty \subset \mathbb{G}_{s'}$ such that $f'_i \to f$ weakly as $i \to \infty$.
\end{Lem}
\begin{proof}
Define
\[
a_s := \sup \{i \geq 1 \, : \, f_i \in \mathbb{G}_s  \}
\]
with $ s < t$. By construction, there exists $s'$ such that $a_{s'} = \infty$ and $\{ f'_i\}$ is the subsequence of elements of $\{ f_i\}$ that belong to $\mathbb{G}_{s'}$.
\end{proof}
\begin{Lem}\label{second_auxiliary}
Let $\left\{ f_i = \sum_{j= 1}^s q_{j,i}k(\cdot \mid \theta_{j,i})\right\}_{i = 1}^\infty \subset \mathbb{G}_s$ be such that $f_i \to f \in \mathcal{P}(\mathbb{X})$ weakly as $i \to \infty$. Then there exist $s' \leq s$ and a sequence $\{ f'_i\}_{i = 1}^\infty \subset \mathbb{G}_{s'}$ such that $f'_i \to f$ weakly as \\$i \to \infty$ and
\[
\lim \inf_i q'_{j,i} > 0
\]
for every $j = 1, \dots, s'$.
\end{Lem}
\begin{proof}
If $\lim \inf_i q_{j,i} = 0$ for every $j = 1, \dots, s$, the statement is true by taking $s := s'$ and $f'_i := f_i$ for every $i \geq 1$. Then assume there exists $l$ such that $\lim \inf_i q_{l,i} = 0$. Consider a subsequence $\{ \tilde{f}_i\}_{i = 1}^\infty$, with weights $\{\tilde{q}_{j,i}\}_i$ and parameters $\{\tilde{\theta}_{j,i}\}_i$, such that $\lim_i \tilde{q}_{l,i} = 0$ and define
\[
f'_i(x) = \sum_{j \neq l} \frac{\tilde{q}_{j,i}}{\sum_{r \neq l}\tilde{q}_{r,i}}k(x \mid \tilde{\theta}_{j,i}),
\]
where $\sum_{r \neq l}\tilde{q}_{r,i} \to 1$, by construction. Let $A \subset \mathbb{X}$, then
\[
\begin{aligned}
\left \lvert \int_A\tilde{f}_i(x) \mu(\d x) - \int_Af'_i(x) \mu(\d x) \right \rvert &=  \sum_{j \neq l}\left(\frac{\tilde{q}_{j,i}}{\sum_{r \neq l}\tilde{q}_{r,i}}- \tilde{q}_{j,i}\right)\int_Ak(x \mid \tilde{\theta}_{j,i}) \mu(\d x) \\
&+\tilde{q}_{l,i}\int_Ak(x \mid \tilde{\theta}_{l,i}) \mu(\d x)\leq \sum_{j \neq l}\left(\frac{\tilde{q}_{j,i}}{\sum_{r \neq l}\tilde{q}_{r,i}}- \tilde{q}_{j,i}\right)+\tilde{q}_{l,i} \to 0,
\end{aligned}
\]
as $i \to \infty$. Therefore, since $A$ is arbitrary and $\{ \tilde{f}_i\}$ converges to $f$, also  $\{ f'_i\}$ converges weakly to $f$ and $\{ f'_i\}_{i = 1}^\infty \in \mathbb{G}_{s-1}$. The result follows by applying recursively the above procedure for every $l$ satisfying $\lim \inf_i q_{l,i} = 0$.
\end{proof}
\begin{proof}[Proof of Proposition \ref{undersmoothing}]% (under the assumptions $H1$--$H3$)]
By Lemma \ref{th:WeakConsDens} we can assume $H1-H4$ and by Lemma \ref{lemma_weak_convergence}, it suffices to prove the existence of a weak neighborhood $\mathcal{U}$ of $P$ such that $\mathcal{U} \cap \mathbb{G}_s = \emptyset$, for every $s < t$. Assume by contradiction that no such $\mathcal{U}$ exists. Then, there exists a sequence $\{f_i\} \in \cap_{s <t}\mathbb{G}_s$ such that $f_i \to f$ weakly, as $i \to \infty$, where $f$ is the density of $P$. By Lemmas \ref{first_auxiliary} and \ref{second_auxiliary} we can assume without loss of generality that $\{f_i\} \in \mathbb{G}_s$, with $s < t$, and $\lim \inf_i q_{j,i} > 0$ for every $j = 1, \dots, s$. We will consider three scenarios, of which at least one must hold: (i) there exists $l \in \{1, \dots, s\}$ such that $\lim \sup_i ||\theta_{l,i}|| = \infty$, (ii) the sequences $\{\theta_{j,i}\}_{i = 1}^\infty$, with $j = 1, \dots, s$, belong to a compact set $C \subset \Theta \backslash \mathbb{B}$ for $i$ large enough, (iii) the sequences $\{\theta_{j,i}\}_{i = 1}^\infty$, with $j = 1, \dots, s$, belong to a compact set $C \subset \Theta$ and there exists $l \in \{1, \dots, s\}$ such that $ \lim \inf _i\inf_{\theta \in \mathbb{B}} ||\theta_{l, i}-\theta || = 0$.

First consider case (i) and assume there exists $1 \leq l \leq s$ such that $|| \theta_{l,r(i)}|| \to \infty$ as $i \to \infty$ for a suitable subsequence $r(i)$. Fix $0 <\epsilon < \lim \inf_i q_{l,i}$ and choose $K \subset \mathbb{X}$ compact set such that $P(K) > 1-\epsilon/4$. By assumption $H2$ we have
\[
\int_{K^c}f_{r(i)}(x) \mu(\d x) > q_{l,r(i)}\int_{K^c}k(x \mid \theta_{l,r(i)}) \mu(\d x) > \frac{\epsilon}{2},
\]
for $i$ large enough, which contradicts the weak convergence of $\{f_i\}_{i = 1}^\infty$ to $f$.

Second, assume to be in case (ii) and there exists a compact set $C \subset \Theta \backslash \mathbb{B}$ such that $\theta_{i,j} \in C$ for every $i \geq 1$ and $j = 1, \dots, s$. Define the set
\[
\mathbb{D}_s : = \left\{\nu(\d \theta) = \sum_{j = 1}^sq_j\delta_{\theta_j}(\d \theta) \, : \, \theta_j \in C, q_j > 0, \sum_{j = 1}^sq_j = 1  \right\} \subset \mathcal{P}(\Theta).
\]
Since $C$ is compact, we have that $\mathbb{D}_s$ is tight. By Prokhorov's Theorem $\mathbb{D}_s$ is also relatively compact, so that there exists a subsequence $r(i)$ such that 
\[
\nu_{r(i)} = \sum_{j = 1}^sq_{j,r(i)}\delta_{\theta_{j,r(i)}} \to \nu \in \mathcal{P}(\Theta)
\]
weakly as $i \to \infty$. By Lemma $4.1$ in \cite{Cai2020} we have $\nu \in \mathbb{D}_{s}$, so that $\nu = \sum_{j = 1}^s\tilde{q}_j\delta_{\tilde{\theta}_j}$ for some $\tilde{q}_j \in (0,1)$, $\sum_{j = 1}^s\tilde{q}_j = 1$ and $\tilde{\theta}_j \in C$, for $j = 1, \dots, s$. By $H1$ and $C \subset \Theta \backslash \mathbb{B}$, for $\mu$-almost every $x \in \mathbb{X}$, we can find $C_j := C_j(x, \tilde{\theta}_j)$, with $j = 1, \dots, s$, closed neighborhood of $\tilde{\theta}_j$, so that $k(x \mid \theta)$ is continuous as a function of $\theta$, with $\theta \in C_j$. Define $D := \left\{\bigcup_{j = 1}^s C_j\right\}\cap C$ compact set: notice that $D \neq \emptyset$, since $\tilde{\theta}_j \in C \cap C_j$, with $j = 1, \dots, s$. Moreover, by construction, the mapping $\theta \in D \to k(x \mid \theta)$ is continuous and therefore bounded, since $D$ is compact. Since $\nu_i \to \nu$ weakly, as $i \to \infty$, there exists $I$ such that for every $i \geq I$ we have $\theta_{j, r(i)} \in D$, for every $j = 1, \dots, s$. Thus, by definition of weak convergence we have
\[
\sum_{j = 1}^sq_{j,r(i)}k(x \mid \theta_{j, r(i)}) = \int k(x \mid \theta) \nu_{r(i)}(\d \theta) \to \int k(x \mid \theta) \nu(\d \theta) = \sum_{j = 1}^s\tilde{q}_{j}k(x \mid \tilde{\theta}_{j}),
\]
as $i \to \infty$. Since almost sure pointwise convergence of densities implies weak convergence, we have
\[
f_{r(i)} \to \tilde{f} = \sum_{j = 1}^s\tilde{q}_{j}k(\cdot \mid \tilde{\theta}_{j})
\]
weakly as $i \to \infty$. By uniqueness of the weak limit, $\tilde{f}(x) = f(x)$ for $\mu$-almost every $x$, that contradicts $H3$.

Third, consider case (iii). Since $\theta_{j, i} \in C \subset \Theta$ compact set, for every $j = 1, \dots, s$ and $i \geq 1$, there exists a suitable subsequence $r(i)$ such that $\theta_{l, r(i)} \to \bar{\theta}$. Since $\mathbb{B}$ is closed by $H1$, we have that $\bar{\theta} \in \mathbb{B}$. By definition of $\mathbb{B}$, this is not possible if $\mu$ is the counting measure, since $k(x \mid \theta) \leq 1$, for every $x \in \mathbb{X}$ and $\theta \in \Theta$. Thus, let $\mu$ be the Lebesgue measure. Then we can fix $\epsilon > 0$ such that
\[
P(B_{x^*}(\epsilon)) < \frac{\lim \inf_i q_{l,i}}{4},
\]
with $x^*$ as in $H2$. Then by $H2$ we have
\[
\int_{B_{x^*}(\epsilon)}f_{r(i)}(x) \mu(\d x) > q_{l,r(i)}\int_{B_{x^*}(\epsilon)}k(x \mid \theta_{l,r(i)}) \mu(\d x) > \frac{\lim \inf_i q_{l,i}}{2},
\]
for $i $ large enough, that again contradicts the weak convergence of $\{f_i\}_{i = 1}^\infty$ to $f$.
\end{proof}

\section{Proof of Theorem \ref{th:ConsUnif}}
The marginal distribution is available and given by the following lemma.
\begin{Lem}\label{marginal_uniform}
Consider $k$ and $q_0$ 
as in \eqref{unif_framework}. Then it holds
\begin{align*}
m(x_{1:n})  &= \frac{2c - \{\max(x_{1:n},\theta^\ast)-\min(x_{1:n},\theta^\ast)\}}{(2c)^{n+1}},
&(x_{1:n}\in[\theta^\ast-c, \theta^\ast+c]^n).
\end{align*}
\end{Lem}

%\subsection{Proof of Lemma \ref{marginal_uniform}}
\begin{proof}%[Proof of Lemma \ref{marginal_uniform}]
Note that $ x_i\in(\theta-c,\theta+c)$ for all $i \in \{1, \dots, n\}$ if and only if $\theta\in(\max(x_{1:n})-c ,\min(x_{1:n})+c)$. Thus
\begin{align*}
m(x_{1:n}) &= \frac{1}{(2c)^{n+1}} \int_{\Theta}\prod_{i = 1}^n \mathbbm{1}_{(\theta-c,\theta+c)}(x_i) \mathbbm{1}_{(\theta^\ast-c,\theta^\ast+c)}(\theta) \mathrm{d}\theta
\\ &=
\frac{1}{(2c)^{n+1}} \int_{\Theta}\mathbbm{1}_{(\max(x_{1:n})-c,\min(x_{1:n})+c)}(\theta) \mathbbm{1}_{(\theta^\ast-c,\theta^\ast+c)}(\theta) \mathrm{d}\theta
\\&=
\frac{2c- \{\max(x_{1:n},\theta^\ast)-\min(x_{1:n},\theta^\ast)\}}{(2c)^{n+1}}\,.
\end{align*}
%    $P^{(\infty)}-$almost surely.
\end{proof}
Define $\text{Range}(X_A) = \max_{i \in A} \left( X_{i}\right)-\min_{i \in A}\left(X_{i} \right)$. Lemma \ref{marginal_uniform} has an important corollary, that is stated after a technical lemma.

\begin{Lem}\label{lemma_unif2}
Let $A \subset \{1, \dots, n\}$ such that $|A| = a$, Then it holds:
\[
\frac{2c- \{\max(X_{A},\theta^\ast)-\min(X_{A},\theta^\ast)\}}{(2c)^{a+1}} \le  \frac{2c-\text{\rm Range}(X_{A})}{(2c)^{a+1}}.
\]
\end{Lem}
\begin{proof}
The result follows immediately from $\max(X_{A},\theta^\ast) \geq \max(X_{A})$ and $\min(X_{A},\theta^\ast) \leq \min(X_{A})$.
\end{proof}
\begin{Cor}\label{main_cor}
In the setting of \eqref{eq:MixDPM} with $(f, k, q_0)$ as in \eqref{unif_framework}, define\\
 $\Omega_n = \left\{ x \in X^{\infty} \, : \,   \max (x_{1:n}) \geq  \theta^\ast \text{ and } \min (x_{1:n} ) \leq \theta^\ast \right\}$. Then
\begin{align}\label{eq:m_ratio_Unif_s+1}
\frac{\prod_{j=1}^{s+1} m(X_{A_j})}{m(X_{1:n})} \mathbbm{1}_{\Omega_n}(X_{1:\infty}) \le \frac{\prod_{j=1}^{s+1} \{2c-\text{\rm Range}(X_{A_j})\}}{(2c)^s \{ 2c-\text{\rm Range}(X_{1:n})\}},
\end{align}
for every $A \in \tau_{s+1}(n)$ .
\end{Cor}
\begin{proof}
As regards the numerator, apply firstly Lemma \ref{marginal_uniform} and then Lemma \ref{lemma_unif2} to get
\[
m(X_{A_j}) = \frac{2c- \{\max(X_{A_j},\theta^\ast)-\min(X_{A_j},\theta^\ast)\}}{(2c)^{a_j+1}} \leq \frac{2c-\text{Range}(X_{A_j})}{(2c)^{a_j+1}}, \quad j = 1, \dots, s+1\,.
\]
Apply Lemma \ref{marginal_uniform} to $m(x_{1:n})$ for every $x \in \Omega_n$, to get
\[
\begin{aligned}
m(X_{1:n})\mathbbm{1}_{\Omega_n}(X_{1:\infty}) &=\frac{2c- \{\max(X_{1:n},\theta^\ast)-\min(X_{1:n},\theta^\ast)\}}{(2c)^{n+1}}\mathbbm{1}_{\Omega_n}(X_{1:\infty}) \\
&= \frac{2c- \{\max(X_{1:n})-\min(X_{1:n})\}}{(2c)^{n+1}}\mathbbm{1}_{\Omega_n}(X_{1:\infty}), 
\end{aligned}
\]
as desired.
\end{proof}
The lemma below shows that, in order to prove Theorem \ref{th:ConsUnif}, it is sufficient to show $\mathbbm{1}_{\Omega_n}(X_{1:\infty})\sum_{s = 1}^{n-1}\frac{\text{pr}(K_n = s+1 | X_{1:n})}{\text{pr}(K_n = 1 | X_{1:n})} \to 0$ in $P^{(\infty)}$-probability.
\begin{Lem}\label{lemma_unif_tech}
Consider $f$ as in \eqref{unif_framework} and define $\Omega_n = \left\{ x \in X^{\infty} \, : \,   \max (x_{1:n}) \geq  \theta^\ast \text{ and } \min (x_{1:n} ) \leq \theta^\ast \right\}$. Let $\left\{Y_n\right\}$ be a sequence of positive random variables. Thus, $Y_n\mathbbm{1}_{\Omega_n}(X_{1:\infty}) \to 0$ in $P^{(\infty)}$-probability implies $Y_n \to 0$ in $P^{(\infty)}$-probability.
\end{Lem}
\begin{proof}
First of all, by definition of $f$ we have
\[
\max (X_{1:n}) \to  \theta^\ast + c, \quad \min (X_{1:n} ) \to \theta^\ast - c
\]
almost surely with respect to $P^{(\infty)}$ as $n \to \infty$. Then $P^{(\infty)}(\Omega_n) \to 1$, as $n \to \infty$, by definition of $\Omega_n$. Thus, fix $\epsilon > 0$ and notice that
\[
P^{(\infty)}\left(Y_n > \epsilon \right) = P^{(\infty)}\left\{\left(Y_n > \epsilon\right) \cap \Omega_n \right\} + P^{(\infty)}\left\{\left(Y_n > \epsilon\right) \cap \Omega_n^c \right\}.
\]
The first term on the right-hand side goes to $0$, since $Y_n\mathbbm{1}_{\Omega_n}(X_{1:\infty}) \to 0$ in $P^{(\infty)}$-probability, while the second vanishes because $P^{(\infty)}(\Omega_n^c) \to 0$, both as $n \to \infty$.
\end{proof}
Combining Corollary \ref{main_cor} and Lemma \ref{lemma_unif_tech} we are ready to prove Theorem \ref{th:ConsUnif}.

\begin{proof}[Proof of Theorem \ref{th:ConsUnif}]
For every $s \geq 1$ and $A \in \tau_s(n)$, from Corollary \ref{main_cor} we have
\[
\frac{\prod_{j=1}^{s} m(X_{A_j})}{m(X_{1:n})} \mathbbm{1}_{\Omega_n}(X_{1:\infty}) \le \frac{\prod_{j=1}^{s} \{2c-\text{Range}(X_{A_j})\}}{(2c)^{s-1} \{ 2c-\text{Range}(X_{1:n})\}}.
\]
Note that $\{2c-\text{Range}(X_{A_j})\}/(2c) \sim \text{Beta}(2,a_j-1)$ independently for $j=1,\ldots,s$. %and $M=1/\{1-\text{Range}(X_{A_j})\}$, where $1/M \sim \text{Be}(2,n-1)$.
Moreover, recall that if $Z \sim \text{Beta}(\alpha,\beta)$ then for $p > - \alpha$
\begin{equation*}
E(Z^p) = \frac{\Gamma(\alpha+p) \Gamma(\alpha+\beta)}{\Gamma(\alpha+p+\beta) \Gamma(\alpha)}.
\end{equation*}
Thus, by H\"{o}lder's inequality with exponents $3$ and $3/2$ we get
\begin{align*}
E\bigg\{\frac{\prod_{j=1}^{s} m(X_{A_j})}{m(X_{1:n})} \bigg\} & \leq E\left\{\prod_{j=1}^{s} m(X_{A_j})^3 \right\}^{1/3}E\left\{ m(X_{1:n}) ^{-3/2}\right\}^{2/3} \\
&= \bigg\{\frac{\Gamma(5)}{\Gamma(2)}\bigg\}^{s/3} \bigg\{\frac{\Gamma(1/2)}{\Gamma(2)}\bigg\}^{2/3} \bigg\{\prod_{j=1}^{s} \frac{\Gamma(1+a_j)}{\Gamma(a_j+4)}\bigg\}^{1/3} \bigg\{\frac{\Gamma(1+n)}{\Gamma (n-1/2 )}\bigg\}^{2/3}.
\end{align*}
By the recursive definition of the Gamma function and recalling that $\Gamma(1/2)=\pi^{1/2}$, the upper bound above becomes
\begin{align*}
E\bigg\{\frac{\prod_{j=1}^{s} m(X_{A_j})}{m(X_{1:n})} \bigg\} &\leq	24^{s/3} \pi^{1/3} \bigg\{\prod_{j=1}^{s} \frac{\Gamma(1+a_j)}{\Gamma(a_j+4)}\bigg\}^{1/3} \bigg\{\frac{\Gamma(1+n)}{\Gamma(n-1/2)}\bigg\}^{2/3}\\
&= 24^{s/3} \pi^{1/3} \bigg\{\prod_{j=1}^{s} \frac{1}{(a_j+3)(a_j+2)(a_j+1)}\bigg\}^{1/3} \bigg\{ \frac{(n-1/2)\Gamma(1+n)}{\Gamma(n+1/2)} \bigg\}^{2/3}.
\end{align*}
%The previous equalities follows by the recursive definition of Gamma function and recalling that $\Gamma(1/2)=\pi^{1/2}$. 
Moreover, exploiting again the recursive definition of the Gamma function, Gautschi's Inequality, i.e.\ $\frac{\Gamma(1+n)}{\Gamma(n+1/2)} \le (n+1)^{1/2}$, and 
%the simple inequalities $n-1/2<n+1/2$ and 
$(n+1)/(a_j+1)<n/a_j$, we have
\begin{align*}
E\bigg\{\frac{\prod_{j=1}^{s} m(X_{A_j})}{m(X_{1:n})} \bigg\} \le 24^{s/3} K \bigg\{\prod_{j=1}^{s} \frac{(n+1)^3}{(a_j+1)^3}\bigg\}^{1/3}
\le 24^{s/3} K  \bigg(\frac{n^3}{\prod_{j=1}^{s} a_i^3} \bigg)^{1/3} = 24^{s/3} K  \frac{n}{\prod_{j=1}^{s} a_j}.
\end{align*}
Thus, applying Lemma \ref{suff_expectation} and Lemma \ref{induction} with $p=2$ and $C=4 \zeta(2) < 7$ we get
\begin{equation*}
E\{R(n,1,s)\} \le \frac{24^{s/3} K}{s!} \sum_{\bm{a} \in \mathcal{F}_s(n)} \bigg(\frac{n}{\prod_{j=1}^s a_j} \bigg)^2 < \frac{C^{s-1} 24^{s/3} K}{s!},
\end{equation*}
where $R(n,1,s)$ is defined as in \eqref{newratios}. From Corollary \ref{cor_C} we have
\[
C(n,1,s+1) \leq \frac{G\Gamma(2+\beta)2^ss}{\epsilon}E(\alpha^s)\log\{n/(1+\epsilon)\}^{-1}, \quad n \geq 4\,.
\]
Thus, combining the inequalities above with \eqref{newratios} and assumption $A3$ we have
\[
\begin{aligned}
E \biggl\{\mathbbm{1}_{\Omega_n}(X_{1:\infty})&\sum_{s = 1}^{n-1}\frac{\text{pr}(K_n = s+1 | X_{1:n})}{\text{pr}(K_n = 1 | X_{1:n})}\biggr\} = \sum_{s = 1}^{n-1}C(n,1,s+1)E\{\mathbbm{1}_{\Omega_n}(X_{1:\infty})R(n,1,s+1)\}\\
&  \leq \frac{24^{1/3}DGK\Gamma(2+\beta)}{\epsilon \log\{n/(1+\epsilon)\}}\underbrace{\sum_{ s = 1}^{n-1}\frac{s (2C24^{1/3})^{s}\rho^{-s}\Gamma(\nu +s+1)}{(s+1)!}}_{< \infty} \to 0\qquad\hbox{as }n\to\infty\,,
\end{aligned}
\]
where finiteness follows from  $\rho \ge 38 > 24^{1/3} \times 2C$. This implies that
\begin{equation*}
\sum_{s = 1}^{n-1}\frac{\text{pr}(K_n = s+1 | X_{1:n})}{\text{pr}(K_n = 1 | X_{1:n})} \rightarrow 0
\end{equation*}
in $L^1$ and thus in $P^{(\infty)}$-probability as $n\to\infty$. Lemma \ref{lemma_unif_tech} with $Y_n = \sum_{s = 1}^{n-1}\frac{\text{pr}(K_n = s+1 | X_{1:n})}{\text{pr}(K_n = 1 | X_{1:n})}$ concludes the proof.
\end{proof}

\section{Proof of Theorem \ref{th:ConsNormDelta}}% (Gaussian mixtures)}
We first need the following result.
\begin{Lem}\label{lemma:m_mratio_Gauss}
Let $k$ and $q_0$ be as in \eqref{eq:NormDelta} and $x_1=\cdots=x_n=\theta^\ast$ for some $\theta^\ast \in\R$. Then
\begin{align*}
\frac{\prod_{j=1}^{s}m(x_{A_j})}{m(x_{1:n})} &= \left\{\frac{n+1}{\prod_{j=1}^{s} (a_j+1)} \right\}^{1/2} \exp\bigg\{\frac{{\theta^\ast}^2}{2} \bigg(- \frac{n^2}{n+1} + \sum_{j=1}^{s} \frac{a_j^2}{a_j+1} \bigg) \bigg\}
< \left(\frac{n}{\prod_{j=1}^{s} a_j} \right)^{1/2}\,,
\end{align*}
for every $s=1,\dots,n$ and every partition $A=\{A_1,\ldots,A_s\} \in \tau_s(n)$.
\end{Lem}
\begin{proof}
Since the marginal likelihood can be rewritten as
\begin{equation*}
m(x_{A_j})= (a_j+1)^{-1/2} q_0(\theta^\ast)^{a_j} \exp\bigg\{\frac{{\theta^\ast}^2}{2} \frac{a_j^2}{a_j+1} \bigg\},
\end{equation*}
the first equality is obtained. The inequality follows from
\[
\begin{aligned}
- \frac{n^2}{n+1} + \sum_{j=1}^{s} \frac{a_j^2}{a_j+1} &= n - \frac{n^2}{n+1} + \sum_{j=1}^{s} \bigg(\frac{a_j^2}{a_j+1} -a_j \bigg) =  \frac{n}{n+1} - \sum_{j=1}^{s} \frac{a_j}{a_j+1} = \\
& = \sum_{j=1}^{s} a_j \bigg(\frac{1}{n+1} -\frac{1}{a_j+1} \bigg) \le 0
\end{aligned}
\]
and
\[
\frac{n+1}{\prod_{j=1}^{s} (a_j+1)} \leq \frac{n}{\prod_{j=1}^{s} a_j},
\]
which easily follows from $a_j \leq n$, for every $j = 1, \dots, s$.
\end{proof}

\begin{proof}[Proof of Theorem \ref{th:ConsNormDelta}]
First, we study $R(n,1,s)$ as defined in \eqref{newratios}. Since all the observations are almost surely equal, we have
\begin{equation*}
R(n,1,s) = \sum_{\bm{a} \in \mathcal{F}_s(n)} \frac{n}{s! \prod_{j=1}^{s} a_j} \frac{\prod_{j=1}^s m(X_{A^{\bm{a}}_j})}{m(X_{1:n})}\,,
\end{equation*}
where $A^{\bm{a}}$ is an arbitrary partition in $\tau_s(n)$ such that $|A_j^{\bm{a}}| = a_j$ for $j=1,\dots,s$. By application of Lemma \ref{lemma:m_mratio_Gauss} and Lemma \ref{induction} with $p=3/2$, it turns out that the constant $C = 2^{\frac{3}{2}}\zeta\left(\frac{3}{2} \right) < 8$ is such that
\begin{equation*}
R(n,1,s) < \frac{1}{s!} \sum_{\bm{a} \in \mathcal{F}_s(n)} \left(\frac{n}{\prod_{j=1}^{s} a_j} \right)^{3/2} < \frac{C^{s-1}}{s!}.
\end{equation*}
From Corollary \ref{cor_C} we have
\begin{equation}\label{eq:thm1_1}
C(n,1,s+1) \leq \frac{G\Gamma(2+\beta)2^ss}{\epsilon}E(\alpha^s)\log\{n/(1+\epsilon)\}^{-1}, \quad n \geq 4\,.
\end{equation}
Thus, combining the inequalities above with \eqref{newratios} and assumption $A3$ we have
\begin{equation}\label{eq:thm1_2}
\begin{aligned}
\sum_{s = 1}^{n-1}\frac{\text{pr}(K_n = s+1 | X_{1:n})}{\text{pr}(K_n = 1 | X_{1:n})} &= \sum_{s = 1}^{n-1}C(n,1,s+1)R(n,1,s+1)\\
&  \leq \frac{DG\Gamma(2+\beta)}{\epsilon \log\{n/(1+\epsilon)\}}\underbrace{\sum_{ s = 1}^{n-1}\frac{s (2C)^{s}\rho^{-s}\Gamma(\nu +s+1)}{(s+1)!}}_{< \infty} \to 0\qquad\hbox{as }n\to\infty\,,
\end{aligned}
\end{equation}
where the finiteness follows from $\rho > 16 > 2C$. Then we conclude applying a variation of Lemma \ref{sufficient} with equalities and limits in probability replaced by almost sure equalities and limits (the proof of Lemma \ref{sufficient} extends trivially to that case).
\end{proof}

\section{Proof of Proposition \ref{prop:PostAlpha}}
\begin{proof}
Under \eqref{eq:MixDPM}, for every $\epsilon > 0$ we have
\begin{align*}
\text{pr}(\alpha < \epsilon \mid X_{1:n}) &= \sum_{s=1}^{n} \text{pr}(\alpha < \epsilon \mid  K_n=s) \,\ \text{pr}(K_n=s \mid X_{1:n}) =\\
&\ge \text{pr}(\alpha < \epsilon \mid K_n=t) \,\ \text{pr}(K_n=t \mid X_{1:n}).
\end{align*}
By assumption, $\text{pr}(K_n=t \mid X_{1:n}) \to 1$ in $P^{(\infty)}$-probability as $n \to \infty$. Moreover, by Proposition \ref{prop:Prior} with $s = 1$ we get
\[
E(\alpha \mid K_n=t) = C(n, t, t+1) \rightarrow 0,
\]
as $n \to \infty$. It follows $\text{pr}(\alpha < \epsilon \mid K_n=t) \to 1$ in $P^{(\infty)}$-probability as $n \to \infty$, as desired.
\end{proof}

\section*{References}

Cai, D., Campbell, T., and Broderick, T. (2021). Finite mixture models do not reliably 
\hspace*{0.2cm} learn the number of components. \emph{38th International Conference on Machine Learning} \hspace*{0.2cm} {\bf{139}}, 1158--1169.\\
\\
Ghosal, S.\ and Van Der Vaart, A.\ W. (2017). \emph{Fundamentals of Nonparametric Bayesian \hspace*{0.2cm} Inference}. Cambridge
University Press.

\end{document}